\documentclass{amsart}

\input{epsf.tex}
\usepackage{amsfonts,amssymb,verbatim,amsmath,amsthm,latexsym,textcomp,amscd, hyperref}
\usepackage{latexsym,amsfonts,amssymb,epsfig,verbatim}
\usepackage{amsmath,amsthm,amssymb,latexsym,graphics,textcomp}
\usepackage{graphicx}
\usepackage{color}
\usepackage{url}

\begin{document}

\newtheorem{theorem}{Theorem}[section]
\newtheorem{prop}[theorem]{Proposition}
\newtheorem{lemma}[theorem]{Lemma}
\newtheorem{cor}[theorem]{Corollary}
\newtheorem{defn}[theorem]{Definition}
\newtheorem{conj}[theorem]{Conjecture}
\newtheorem{claim}[theorem]{Claim}
\newtheorem{rem}[theorem]{Remark}
\newtheorem{rmk}[theorem]{Remark}

\newcommand{\map}{\rightarrow}
\newcommand{\boundary}{\partial}
\newcommand{\C}{{\mathbb C}}
\newcommand{\lct}{\Lambda_{CT}}
\newcommand{\lel}{\Lambda_{EL}}
\newcommand{\lgel}{\Lambda_{GEL}}
\newcommand{\lre}{\Lambda_{\mathbb{R}}}
\newcommand{\integers}{{\mathbb Z}}
\newcommand{\natls}{{\mathbb N}}
\newcommand{\ratls}{{\mathbb Q}}
\newcommand{\reals}{{\mathbb R}}
\newcommand{\proj}{{\mathbb P}}
\newcommand{\lhp}{{\mathbb L}}
\newcommand{\tube}{{\mathbb T}}
\newcommand{\cusp}{{\mathbb P}}
\newcommand\AAA{{\mathcal A}}
\newcommand\BB{{\mathcal B}}
\newcommand\CC{{\mathcal C}}
\newcommand\DD{{\mathcal D}}
\newcommand\EE{{\mathcal E}}
\newcommand\FF{{\mathcal F}}
\newcommand\GG{{\mathcal G}}
\newcommand\HH{{\mathcal H}}
\newcommand\II{{\mathcal I}}
\newcommand\JJ{{\mathcal J}}
\newcommand\KK{{\mathcal K}}
\newcommand\LL{{\mathcal L}}
\newcommand\MM{{\mathcal M}}
\newcommand\NN{{\mathcal N}}
\newcommand\OO{{\mathcal O}}
\newcommand\PP{{\mathcal P}}
\newcommand\QQ{{\mathcal Q}}
\newcommand\RR{{\mathcal R}}
\newcommand\SSS{{\mathcal S}}
\newcommand\TT{{\mathcal T}}
\newcommand\UU{{\mathcal U}}
\newcommand\VV{{\mathcal V}}
\newcommand\WW{{\mathcal W}}
\newcommand\XX{{\mathcal X}}
\newcommand\YY{{\mathcal Y}}
\newcommand\ZZ{{\mathcal Z}}
\newcommand\CH{{\CC\HH}}
\newcommand\TC{{\TT\CC}}
\newcommand\EXH{{ \EE (X, \HH )}}
\newcommand\GXH{{ \GG (X, \HH )}}
\newcommand\GYH{{ \GG (Y, \HH )}}
\newcommand\PEX{{\PP\EE  (X, \HH , \GG , \LL )}}
\newcommand\MF{{\MM\FF}}
\newcommand\PMF{{\PP\kern-2pt\MM\FF}}
\newcommand\ML{{\MM\LL}}
\newcommand\PML{{\PP\kern-2pt\MM\LL}}
\newcommand\GL{{\GG\LL}}
\newcommand\Pol{{\mathcal P}}
\newcommand\half{{\textstyle{\frac12}}}
\newcommand\Half{{\frac12}}
\newcommand\Mod{\operatorname{Mod}}
\newcommand\Area{\operatorname{Area}}
\newcommand\ep{\epsilon}
\newcommand\hhat{\widehat}
\newcommand\Proj{{\mathbf P}}
\newcommand\U{{\mathbf U}}
 \newcommand\Hyp{{\mathbf H}}
\newcommand\D{{\mathbf D}}
\newcommand\Z{{\mathbb Z}}
\newcommand\R{{\mathbb R}}
\newcommand\Q{{\mathbb Q}}
\newcommand\E{{\mathbb E}}
\newcommand\til{\widetilde}
\newcommand\length{\operatorname{length}}
\newcommand\tr{\operatorname{tr}}
\newcommand\gesim{\succ}
\newcommand\lesim{\prec}
\newcommand\simle{\lesim}
\newcommand\simge{\gesim}
\newcommand{\simmult}{\asymp}
\newcommand{\simadd}{\mathrel{\overset{\text{\tiny $+$}}{\sim}}}
\newcommand{\ssm}{\setminus}
\newcommand{\diam}{\operatorname{diam}}
\newcommand{\pair}[1]{\langle #1\rangle}
\newcommand{\T}{{\mathbf T}}
\newcommand{\inj}{\operatorname{inj}}
\newcommand{\pleat}{\operatorname{\mathbf{pleat}}}
\newcommand{\short}{\operatorname{\mathbf{short}}}
\newcommand{\vertices}{\operatorname{vert}}
\newcommand{\collar}{\operatorname{\mathbf{collar}}}
\newcommand{\bcollar}{\operatorname{\overline{\mathbf{collar}}}}
\newcommand{\I}{{\mathbf I}}
\newcommand{\tprec}{\prec_t}
\newcommand{\fprec}{\prec_f}
\newcommand{\bprec}{\prec_b}
\newcommand{\pprec}{\prec_p}
\newcommand{\ppreceq}{\preceq_p}
\newcommand{\sprec}{\prec_s}
\newcommand{\cpreceq}{\preceq_c}
\newcommand{\cprec}{\prec_c}
\newcommand{\topprec}{\prec_{\rm top}}
\newcommand{\Topprec}{\prec_{\rm TOP}}
\newcommand{\fsub}{\mathrel{\scriptstyle\searrow}}
\newcommand{\bsub}{\mathrel{\scriptstyle\swarrow}}
\newcommand{\fsubd}{\mathrel{{\scriptstyle\searrow}\kern-1ex^d\kern0.5ex}}
\newcommand{\bsubd}{\mathrel{{\scriptstyle\swarrow}\kern-1.6ex^d\kern0.8ex}}
\newcommand{\fsubeq}{\mathrel{\raise-.7ex\hbox{$\overset{\searrow}{=}$}}}
\newcommand{\bsubeq}{\mathrel{\raise-.7ex\hbox{$\overset{\swarrow}{=}$}}}
\newcommand{\tw}{\operatorname{tw}}
\newcommand{\base}{\operatorname{base}}
\newcommand{\trans}{\operatorname{trans}}
\newcommand{\rest}{|_}
\newcommand{\bbar}{\overline}
\newcommand{\UML}{\operatorname{\UU\MM\LL}}
\newcommand{\EL}{\mathcal{EL}}
\newcommand{\tsum}{\sideset{}{'}\sum}
\newcommand{\tsh}[1]{\left\{\kern-.9ex\left\{#1\right\}\kern-.9ex\right\}}
\newcommand{\Tsh}[2]{\tsh{#2}_{#1}}
\newcommand{\qeq}{\mathrel{\approx}}
\newcommand{\Qeq}[1]{\mathrel{\approx_{#1}}}
\newcommand{\qle}{\lesssim}
\newcommand{\Qle}[1]{\mathrel{\lesssim_{#1}}}
\newcommand{\simp}{\operatorname{simp}}
\newcommand{\vsucc}{\operatorname{succ}}
\newcommand{\vpred}{\operatorname{pred}}
\newcommand\fhalf[1]{\overrightarrow {#1}}
\newcommand\bhalf[1]{\overleftarrow {#1}}
\newcommand\sleft{_{\text{left}}}
\newcommand\sright{_{\text{right}}}
\newcommand\sbtop{_{\text{top}}}
\newcommand\sbot{_{\text{bot}}}
\newcommand\sll{_{\mathbf l}}
\newcommand\srr{_{\mathbf r}}
\newcommand\geod{\operatorname{\mathbf g}}
\newcommand\mtorus[1]{\boundary U(#1)}
\newcommand\A{\mathbf A}
\newcommand\Aleft[1]{\A\sleft(#1)}
\newcommand\Aright[1]{\A\sright(#1)}
\newcommand\Atop[1]{\A\sbtop(#1)}
\newcommand\Abot[1]{\A\sbot(#1)}
\newcommand\boundvert{{\boundary_{||}}}
\newcommand\storus[1]{U(#1)}
\newcommand\Momega{\omega_M}
\newcommand\nomega{\omega_\nu}
\newcommand\twist{\operatorname{tw}}
\newcommand\modl{M_\nu}
\newcommand\MT{{\mathbb T}}
\newcommand{\Teich}{Teich\-m\"u\-ller\ }
\renewcommand{\Re}{\operatorname{Re}}
\renewcommand{\Im}{\operatorname{Im}}

\title{Algebraic Ending Laminations and Quasi-convexity}

\author{Mahan Mj}
\address{School of Mathematics, Tata Institute of Fundamental Research, 
  Homi Bhabha Road, Mumbai-400005, India}
\email{mahan.mj@gmail.com; mahan@math.tifr.res.in}

\author{Kasra Rafi}
\address{Department of Mathematics, University of Toronto, 4 St. George, Toronto, Canada}
\email{kasra.rafi@gmail.com; rafi@math.toronto.edu}

\thanks{Research of first author partially supported by a Department of Science and Technology J C Bose Fellowship. Research of second author partially supported 
NSERC RGPIN-435885.}
\date{\today}

\begin{abstract}
We explicate a number of notions of algebraic laminations existing in the literature, particularly in the context of an exact sequence
$$1\to H\to G \to Q \to 1 $$ of hyperbolic groups.
These laminations arise in different  contexts: existence of Cannon-Thurston maps; closed geodesics exiting ends of manifolds; dual to actions on
$\R-$trees.

We use the relationship between these laminations to prove quasi-convexity results for finitely generated infinite index subgroups of $H$,
the normal subgroup in the exact sequence above. 
\end{abstract}

\maketitle

\tableofcontents
\section{Introduction} 
\subsection{Statement of results} The main results in this paper establish that for certain naturally occurring  distorted (in the sense of Gromov \cite{gromov-ai}) hyperbolic subgroups $H$ of hyperbolic groups $G$, many quasiconvex subgroups $K$ of $H$ are in fact quasiconvex in the larger  hyperbolic  group $G$. 
The following, which is one of the main theorems of this paper, illustrates this.

\begin{theorem} (See Theorems \ref{surf-coco} and \ref{free-coco})\\
Let $$ 1 \to H \to G \to Q \to 1$$ be an exact sequence of hyperbolic groups, where $H$ is either a free group or a (closed) surface group and $Q$ is convex cocompact
in Outer Space or \Teich space respectively (for the free group, we assume further that $Q$ is purely hyperbolic).
Let $K$ be  a finitely generated  infinite index subgroup  of $H$. Then $K$ is quasi-convex in $G$. \label{omni1} \end{theorem}

The (original motivating) case where $H$ is a closed surface group and $Q = \Z$ in Theorem \ref{omni1} was dealt with by Scott and Swarup in \cite{scottswar}.
The more general case of $H$ a closed surface group (and no further restrictions on $Q$) 
 was obtained by Dowdall, Kent and Leininger recently in \cite{dkl} by different methods. 
 In the preprint \cite{dt2}, which appeared shortly after a first version of the present paper was made public,
Dowdall and Taylor use the methods of their earlier work \cite{dt1} on convex cocompact purely hyperbolic subgroups of $Out(F_n)$ to give a substantially different
proof of Theorem \ref{omni1} when $H$ is free.

For the statement of our next theorem, some terminology needs to be introduced. A \Teich geodesic ray
$r (\subset Teich (S))$ is said to be {\bf thick} \cite{minsky-top, minsky-jams, minsky-bddgeom, rafi-gt}
if $r$ lies in the thick part of \Teich space, i.e. there exists $\epsilon > 0$ such that $\forall x \in r$, the length of the shortest closed geodesic (or
equivalently,  injectivity radius for closed surfaces) on the hyperbolic surface $S_x$ corresponding to $x \in  Teich (S)$ is bounded below by $\epsilon$.
It follows (from \cite{masur-minsky2,minsky-bddgeom, rafi-gt}) that the projection of $r$ to the curve complex is a parametrized quasi-geodesic and the universal curve $U_r$
over $r$ (associating $S_x$ to $x$ and equipping the resulting bundle with an infinitesimal product metric) has a hyperbolic universal cover $\til U_r$ \cite{minsky-jams, minsky-bddgeom}.
To emphasize this hyperbolicity we shall call these geodesic rays {\bf thick hyperbolic rays}. We shall refer to $\til U_r$ as the {\bf universal metric bundle} (of hyperbolic planes)
 over $r$.

Analogously, we define a geodesic ray $r$ in Culler-Vogtmann outer space $cv_n$ \cite{cv} to be  {\bf thick hyperbolic} if
\begin{enumerate}
\item $r$ projects to a parametrized quasi-geodesic in the free factor complex $\FF_n$.
\item the bundle of trees $X$ over $r$ (thought of as a metric bundle \cite{mahan-sardar}, see Section
\ref{sec:mbdl} below) is hyperbolic.
\end{enumerate}
In this  case too, we shall refer to $X$ as the universal metric bundle (of trees) over $r$. 

\begin{theorem} (See Theorems \ref{surf-ray} and \ref{free-ray})\\
Let $r$ be a thick hyperbolic quasi-geodesic ray
\begin{enumerate}
\item either in $Teich(S)$ for $S$ a closed surface of genus greater than one  
\item or in the Outer space $cv_n$ corresponding to $F_n$.
\end{enumerate} 
Let $X$ be the universal metric bundle of hyperbolic planes or trees (respectively) over $r$.
Let $H $ denote respectively $\pi_1(S)$ or $F_n$ and $i: H \to X$ denote an orbit map.
Let $K$ be  a finitely generated  infinite index subgroup  of $H$. Then $i(K)$ is quasi-convex in $X$. \label{omni2} \end{theorem}

The following Theorem generalizes the closed surface cases of Theorems \ref{omni1} and \ref{omni2} to surfaces with punctures.

\begin{theorem}\label{omni3} (See Theorems \ref{punct-coco} and \ref{punct-ray})\\
 Let $H=\pi_1(S^h)$ for $S^h$ a hyperbolic surface of finite volume.
Let $r$ be a thick hyperbolic ray in \Teich space $Teich(S^h)$ and let $r_\infty \in \partial Teich(S^h)$ be the limiting surface ending lamination.
Let $X$ denote the universal metric bundle over $r$ minus a small neighborhood of the cusps and let $\HH$ denote the horosphere
boundary components. Let $K$ be a finitely generated infinite index subgroup of $H$. Then any orbit of $K$ in $X$  is relatively
quasi-convex in $(X, \HH)$.\\

Let $H = \pi_1(S^h)$ be the fundamental group of a surface with
finitely many punctures and let  $H_1, \cdots ,H_n$ be its peripheral subgroups. Let $Q$ be a
convex cocompact subgroup of the pure mapping class group of $S^h$. Let
$$1 \to H\to G \to
Q  \to 1$$
and
$$1 \to H_i \to N_G(H_i) \to
Q \to 1$$
be the induced short exact sequences of groups. Then $G$ is strongly hyperbolic relative
to the collection  $ \{N_G(H_i)\}, i = 1, \cdots , n$.

 Let $K$ be a finitely generated infinite index subgroup of $H$. Then  $K$ 
is  relatively quasi-convex  in $G$.
\end{theorem}

The first part of the second statement in Theorem \ref{omni3} is from \cite{mahan-sardar}. The relative quasi-convexity 
part  of the second statement (which requires relative hyperbolicity as its
framework) is what is new.

\subsection{Techniques} The main technical tool used to establish the above theorems is the theory of laminations. A guiding motif that underlies much of this paper is that the directions of maximal distortion for a hyperbolic group $H$ acting on a hyperbolic metric space $X$ are encoded in a lamination. Hence if the set of such laminations supported on a subgroup $K$ of $H$ is empty, we should expect that the subgroup $K$  is undistorted in $X$, or equivalently, quasiconvex
in $X$. Unfortunately, there are a number of  competing notions of laminations existing in the literature; and they do not all serve the same purpose. To make this philosophy work therefore, we need to investigate the relationships between these different kinds of laminations.

The weakest notion is that of
an {\bf algebraic lamination} \cite{bfh-lam, chl07,  kl10, kl15, mitra-endlam} for a hyperbolic group $H$: an
$H$-invariant, flip invariant, closed subset 
$$
\LL \subseteq \partial^2 H =(\partial H \times \partial H \setminus \Delta),
$$ 
where  $\Delta$ denotes the diagonal in $\partial H \times \partial H$. 

Several classes of algebraic laminations have come up in the study of automorphisms of hyperbolic groups, especially free and surface groups:

\begin{enumerate}
	\item The dual lamination $\lre$ arising from an action of $H$ on an $\R-$tree \cite{thurstonnotes, bfh-lam, chl07, chl08a,  kl10}. See Definition \ref{def-lre} which allows us to make sense of this for the action of any hyperbolic group  $H$ on an $\R-$tree.
	\item The ending lamination $\lel$ or $\lgel$ arising from closed geodesics
	exiting an end of a 3-manifold  \cite{thurstonnotes} (see also \cite{mitra-endlam} for an algebraization 
	of this concept). In the (group-theoretic) context of this paper, $\lel$ or $\lgel$ is defined using Gromov-Hausdorff limits following \cite{mitra-endlam} rather than projectivized measured lamination space as in  \cite{thurstonnotes}. Thus,  $\lel$ or $\lgel$ may be intuitively described as Hausdorff limits of closed curves
	whose geodesic realizations exit an end. For normal hyperbolic subgroups $H$ of hyperbolic groups $G$, $\lel = \lgel$ \cite{mitra-endlam}.
	\item The Cannon-Thurston lamination $\lct$ arising in the context of the existence of a Cannon-Thurston map \cite{CT,CTpub,mitra-endlam}.
\end{enumerate}

Note that the above three notions make sense in the rather general context of a hyperbolic group $H$.  
These different kinds of laminations play different roles.
\begin{enumerate}
	\item 
	The dual lamination $\lre$ often has good mixing properties like arationality \cite{thurstonnotes} or minimality \cite{chr} or the dual notion of indecomposability
	for the dual $\R-$tree
	\cite{gui}. 
	\item The Cannon-Thurston laminations $\lct$ play a role in determining quasi-convexity of subgroups \cite{scottswar, mitra-pams}. See Lemma \ref{qccrit} below.
	\item The above two quite different contexts are mediated by ending laminations $\lel$  in the following sense. Theorem \ref{ct=el} \cite{mitra-endlam} equates $\lel$ with $\lct$ in the general context of hyperbolic normal subgroups of hyperbolic groups. The relationship between $\lel$ and $\lre$ has not been established in this generality. It is known however for surface groups \cite{minsky-bddgeom, mitra-endlam} and free groups
	\cite{kdt} in the context of convex cocompact subgroups of the mapping class group or $Out(F_n)$. It is this state of the art with respect to the relationship between $\lel$ and $\lre$ that forces us to restrict ourselves to surface groups and free groups in this paper.
\end{enumerate}

We give a few forward references to indicate  how $\lel$ mediates between $\lct$ and $\lre$ and also sketch the strategy of proof of the main results.
It is easy to see that in various natural contexts the collection of ending laminations $\lel$ or $\lgel$ are contained in the collection of Cannon-Thurston laminations $\lct$
(Proposition \ref{ctgel} below) 
as well as in the dual laminations $\lre$  (Proposition \ref{lregel} below). 
Further, the (harder) reverse containment of $\lct$ in $\lel$ has been established in a number of cases (Theorem \ref{ct=el} below from \cite{mitra-endlam} 
for instance).
What remains is to examine the reverse containment of $\lre$ in $\lel$ in order to complete the picture. This is the subject of \cite{kl15, kdt}
in the context of free groups and \cite{mahan-elct} in the context of surface Kleinian groups.

What kicks in after this are the mixing properties of $\lre$ established by various
authors. Arationality of ending laminations for surface groups was established in
\cite{thurstonnotes} and  arationality in a strong form for free groups was established in
\cite{preynolds-gd, preynolds-red, br-bdy,
	chr, gui}.  It follows from these results that $\lct$ is arational in a strong sense--no leaf of $\lct$ is contained in a finitely generated  infinite index subgroup $K$ of $H$
for various specific instances of $H$. Quasi-convexity of $K$ in $G$ (or more generally some hyperbolic metric bundle $X$) then follows from Lemma \ref{qccrit}.
Accordingly each of the Sections \ref{surf}, \ref{free} and \ref{punct} has two subsections each: one establishing arationality and the second combining arationality along with the 
general theory of Section  \ref{lam} to prove quasi-convexity.

\section{Cannon-Thurston Maps and Metric Bundles}
\subsection{Cannon-Thurston Maps}  Let $H$ be a hyperbolic subgroup of a hyperbolic group $G$
(resp. a group $H$ acting properly on a hyperbolic metric space $X$). Let $\Gamma_H$, $\Gamma_G$ denote
Cayley graphs of $H$, $G$ with respect to finite generating sets. Assume that the finite generating set 
for $H$ is contained in that of $G$. 
 Let $\widehat{\Gamma_H}$,
 $\widehat{\Gamma_G}$ and $\widehat X$ denote the Gromov compactifications. Further let $\partial{H}$,
$\partial{G}$ and $\partial X$ denote the boundaries  \cite{gromov-hypgps}. (It is a fact that 
the boundaries $\partial{\Gamma_H}$ or $\partial{\Gamma_G}$ of the corresponding Cayley graphs
is independent of the finite generating sets chosen; hence we use the
notations $\partial{H}$ and $\partial{G}$).

\begin{defn} Let $H$ be a hyperbolic subgroup of a hyperbolic group $G$
(resp. acting properly on a hyperbolic metric space $X$). 
Let $\Gamma_H, \Gamma_G$ denote  Cayley graphs of $H, G$ as above.

 Let 
 $i : \Gamma_H \rightarrow \Gamma_G$ (resp.  $i : \Gamma_H \rightarrow X$) denote the inclusion
map (resp. an orbit map of $H$  extended by means of geodesics over edges).

A  Cannon-Thurston map for the pair
$(H,G)$ (resp. $(H,X)$) is said to exist if there exists a continuous
extension of $i$  to
$\hat{i} : \widehat{\Gamma_H} \rightarrow \widehat{\Gamma_G}$ (resp.
$\hat{i} : \widehat{\Gamma_H} \rightarrow \widehat{X}$). The restriction $\partial i: \partial{H} \rightarrow \partial{G}$ (resp.
$\partial{i} : \partial{H} \rightarrow \partial{X}$) of $\hat i$ is then called the Cannon-Thurston map for the pair
$(H,G)$ (resp. $(H,X)$). \end{defn}

\begin{theorem}\cite{mitra-ct}\label{ctexist}
 Let $G$ be a hyperbolic group and let $H$ be a hyperbolic normal subgroup
of $G$. Then a Cannon-Thurston map exists for the pair $(H,G)$.
\end{theorem}

\subsection{Metric Bundles}\label{sec:mbdl} To state a theorem analogous to Theorem \ref{ctexist} in the more general 
geometric (not necessarily group-invariant) setting
of a metric bundle, some material needs to be summarized from \cite{mahan-sardar}.

\begin{defn}\label{defn-mbdle}
	Suppose $(X,d)$ and $(B, d_B)$ are geodesic metric spaces; let $c\geq 1$ be a constant and let 
	$f:{\mathbb R}^+ \rightarrow {\mathbb R}^+$ be a function.
	We say that $X$ is an $(f,c)-$ {\bf metric bundle} over $B$ if there is a surjective $1$-Lipschitz
	map $p:X\rightarrow B$ such that the following conditions hold:
	\begin{enumerate}
		\item For each point $z\in B$, $F_z:=p^{-1}(z)$ is a geodesic metric space
		with respect to the path metric $d_z$ induced from $X$. The inclusion maps
		$i: (F_z,d_z) \rightarrow X$ are uniformly metrically proper as measured by $f$,
		i.e.  for all $z\in B$ and $x,y\in F_z$,  $d(i(x),i(y))\leq N$ implies that $d_z(x,y)\leq f(N)$.
		\item Suppose $z_1,z_2\in B$, $d_B(z_1,z_2)\leq 1$ and let $\gamma$ be
		a geodesic in $B$ joining them. 
		 Then for any point $x\in F_z$, $z\in \gamma$, there is a path in $p^{-1}(\gamma)$
		of length at most $c$ joining $x$ to both $F_{z_1}$ and $F_{z_2}$.		

		It follows that there exists $K=K(f,c) \geq 1$, such that the following holds:	
		Suppose $z_1,z_2\in B$ with $d_B(z_1,z_2)\leq 1$ and let $\gamma$ be
		a geodesic in $B$ joining them.
		Let $\phi: F_{z_1}\rightarrow F_{z_2}$, be any map such that
		$\forall x_1\in F_{z_1}$ there is a path of length at most $c$ in $p^{-1}(\gamma)$
		joining $x_1$ to $\phi(x_1)$. Then $\phi$ is a $K$-quasi-isometry.
	\end{enumerate}
\end{defn}

We now describe the two kinds of metric bundles that will concern us in this paper. First, let
$r$ be a geodesic (or more generally a quasigeodesic) ray in $Teich(S)$ for $S$ a closed surface of genus greater than one. Then the 
universal bundle over $Teich(S)$ restricted to $r$, $U_r$ say,
 has a natural metric.  Through any point $x \in S_z$,
the fiber over $z \in Teich(S)$, there is a canonical isometric lift of $r$. 
By declaring these lifts to be orthogonal to $S_z$ at every such point $x$ equips $U_r$ with the natural
metric. The universal cover of $U_r$ with the lifted metric is then the required  metric bundle
over $r$. The  fiber $F_z$ over  $z$ is the universal cover of $S_z$.

The (unprojectivized) Culler-Vogtmann Outer
space corresponding to $F_n$ will be denoted by $cv_n$ \cite{cv} and its boundary
by $\partial cv_n$.
We describe the metric bundle over a ray $r$ in  $cv_n$. For our purposes  we shall  require that $r$
is a folding path \cite{bf2}. It is proved in \cite[Proposition 2.5]{bf2} that for any $z \in cv_n$, there
is a point $z'$ at uniformly bounded distance from $z$, such that a geodesic ray starting at $z$ may be
constructed as a concatenation of a geodesic segment from $z$ to $z'$ followed by a  folding path starting
at $z'$. Thus, for our purposes, up to changing the initial point of $r$ by a uniformly bounded amount,
we might as well assume that $r$ is a folding path. The universal (marked) graph bundle  over
$cv_n$ restricted to $r$, $U_r$ say,  is,  as before equipped with a natural
metric by lifting $r$  isometrically to geodesic rays
through  points in fibers. The universal cover of this bundle of graphs with the lifted metric is
the metric bundle over $r$ in this situation. Note that since folding paths define maps between fibers over
two points in a natural way, the resulting metric bundle comes canonically equipped with an action
of a free group acting fiber-wise.

The next  theorem establishes the existence of a Cannon-Thurston map in this setting:

\begin{theorem}\cite[Theorem 5.3]{mahan-sardar} \label{ctexist2}
Let $r$ be one of the following: 
\begin{enumerate}
\item a thick hyperbolic quasi-geodesic ray in $Teich(S)$ for $S$ a closed surface of genus greater than one  
\item a folding path in the Outer space $cv_n$ corresponding to $F_n$.
\end{enumerate} 
Let $X$ be the universal metric bundle of hyperbolic planes or trees (respectively) over $r$ and suppose that $X$ is hyperbolic.
Let $H $ denote respectively $\pi_1(S)$ or $F_n$. Then the pair $(H, X)$ has a Cannon-Thurston map.
\end{theorem}

The paper \cite{mahan-sardar} deals with  a somewhat more general notion (referred to in \cite{mahan-sardar} as a metric graph bundle)  than the one covered by Definition \ref{defn-mbdle}. However, for the purposes of this paper, it suffices to consider the more restrictive notion of a metric bundle given by \ref{defn-mbdle}. Theorem \ref{ctexist2} in the form that we shall apply it will require only the restricted notion of Definition \ref{defn-mbdle}.
 
\section{Laminations}\label{lam}
An {\bf algebraic lamination} \cite{bfh-lam, chl07,  kl10, kl15, mitra-endlam} for a hyperbolic group $H$ is an
$H$-invariant, flip invariant, non-empty closed subset 
$$
\LL \subseteq \partial^2 H =(\partial H \times \partial H \setminus \Delta),
$$ 
where $\Delta$ is the diagonal in $\partial H \times \partial H$ and   the flip is given by $(x,y)\sim(y,x)$. Here $\partial H$ is equipped with the Gromov topology, and $ \partial^2
H $ with the subspace topology of the product topology. (Note that in \cite{gromov-hypgps}, the notation $\partial^2
H$ is reserved for $(\partial H \times \partial H \setminus \Delta)/\sim$. We prefer to use the notation here as we shall generally be dealing with bi-infinite geodesics rather than unordered pairs of points on $\partial H$.)
Various classes of laminations exist in the literature and in this section, we describe three such classes that arise naturally.

\subsection{Cannon-Thurston Laminations} In this section we shall define laminations in the context of a hyperbolic group $H$
acting properly on a hyperbolic metric space $X$. For instance, $X$ could be a Cayley graph of   a hyperbolic group $G$
containing $H$. We choose, as before, a generating set of $H$, and in case $X$ is a 
Cayley graph of   a hyperbolic supergroup $G$, we assume that the generating set of $H$ is extended to one of $G$,
ensuring a natural inclusion map $i: \Gamma_H \to \Gamma_G$. 
 Choosing a base-point $*$, the orbit map from the vertex set of $H$ to $X$, sending $h$ to $h*$
will be denoted by $i$. Further $i$ is extended to the edges of $\Gamma_H$ by sending them to geodesic segments
in $X$. The laminations we consider in this section go back to \cite{mitra-endlam}
and correspond intuitively to (limits) of geodesic segments in $H$ whose geodesic realizations in $X$ live outside large balls about 
a base-point.

We recall some basic facts and notions (cf. \cite{mitra-endlam, mitra-pams}).
If $\lambda$ is a geodesic segment in $\Gamma_H$,  a {\bf geodesic realization}
 $\lambda^r$, of $\lambda$, is a geodesic in  $X$ joining the end-points of $i({\lambda})$. 

Let  $\{\lambda_n \}_n \subset \Gamma_H$ be a sequence of geodesic segments
such that $1\in{\lambda_n}$ and
${\lambda_n^r}\cap{B(n)} = \emptyset$, where $B(n)$ is the
ball of radius $n$ around $i(1) \in X$. Take all bi-infinite
subsequential limits of pairs of end-points of all such sequences $\{ \lambda_i \}$
(in the product topology on  $\widehat{\Gamma_H} \times \widehat{\Gamma_H}$) 
 and denote this set by $\LL_0$. 
Let $t_h$ denote left translation by $h \in H$.

\begin{defn} The  Cannon-Thurston  pre-lamination
$\Lambda_{CT} = \Lambda_{CT} (H, X)$ is given by 
\begin{center}
$\lct = \Big\{ \{p,q\} \in \partial^2{H}:
p,q$ are the end-points of ${t_h}({\lambda})$ for some 
$\lambda \in \LL_0 \Big\}$
\end{center} 
\end{defn}

For the definition of $\lct$ above, one does not need the existence of a Cannon-Thurston map. However,
$\Lambda_{CT}$ above is not yet a lamination, as closedness is not guaranteed (as was pointed out to us by the referee); hence the expression pre-lamination. In the presence of a Cannon-Thurston map, $\Lambda_{CT}$ is indeed a lamination and we have an alternate description of $\lct$ as follows.

\begin{defn} Suppose that a Cannon-Thurston map exists for the pair $(H,X)$.
We define
$$
\Lambda_{CT}^1 = 
 \big\{ \{p,q\} \in \partial^2{H}  : 
\hat{i} (p) = \hat{i} (q) \big\}.
$$
\end{defn}

\begin{lemma}\cite{mitra-pams}
If a Cannon-Thurston map exists, $\lct= \lct^1$ is a lamination.
\label{equality}
\end{lemma}

Note that closedness of  $\lct$ follows from continuity of the Cannon-Thurston map. 
The following Lemma characterizes quasi-convexity in terms of $\lct$.

\begin{lemma}\cite{mitra-pams}
$H$ is quasi-convex in $X$ if and only if $\lct = \emptyset$
\label{qccrit}
\end{lemma}

We shall be requiring a generalization of Lemma \ref{qccrit} to relatively hyperbolic groups \cite{gromov-hypgps, farb-relhyp, bowditch-relhyp}.
 Let $H $ be a relatively hyperbolic group, hyperbolic relative to a finite collection of parabolic subgroups $\PP$. The relative hyperbolic (or Bowditch) boundary $\partial ( H, \PP)
=\partial_r H$
of the relatively hyperbolic group  $(H, \PP)$ was defined by Bowditch \cite{bowditch-relhyp}.
 The collection of bi-infinite geodesics  $\partial^2_r H$ is given by  $(\partial_r H \times \partial_r H \setminus \Delta)$ as usual.
The existence of a Cannon-Thurston map in this setting of a relatively hyperbolic group $H$ acting on a relatively hyperbolic space $(X, \HH)$
has been investigated in \cite{bowditch-ct, brahma-bddgeo, mj-pal}.
Such an $H$ acts in a {\bf strictly type preserving} manner on a relatively hyperbolic space $(X, \HH)$ if the stabilizer $Stab_H(Y) $ for any $Y \in \HH$ is 
equal to a conjugate of an element
of $\PP$ and if each conjugate  of an element
of $\PP$ stabilizes some  $Y \in \HH$.  The notion of the Cannon-Thurston lamination $\lct  = \lct (H,X)$ is defined as above to be the set of pairs of distinct points
 $\{x,y\} \in \partial^2_r H $ identified by the Cannon-Thurston map.
 The proof of  Lemma \ref{qccrit} from \cite{mitra-pams} directly translates to the following in the relatively
 hyperbolic setup. We refer the reader to \cite{hruska-wise} for the definition of relative quasi-convexity.

\begin{lemma}Suppose that the relatively hyperbolic group $(H, \PP)$ acts in a {\bf strictly type preserving} manner on a relatively hyperbolic space $(X, \HH)$
such that the pair $(H, X)$ has a Cannon-Thurston map.  Let $\lct  = \lct (H,X)$. Then
any orbit of $H$ is relatively quasi-convex in $X$ if and only if $\lct = \emptyset$.
\label{qccritrh}
\end{lemma}

\begin{rmk}
We include an observation as to what happens when we pass to quasi-convex 
or relatively quasi-convex subgroups. Let $K$ (resp.$(K, \PP_1)$)
be a 
quasi-convex (resp. relatively quasi-convex) subgroup of a hyperbolic (resp. relatively hyperbolic) 
group $H$ (resp. $(H, \PP)$).  Then the boundary $\partial K$ (resp. $\partial_r K$) embeds in 
$\partial H$ (resp. $\partial_r H$). This induces an embedding of $\partial^2K$ 
(resp. $\partial^2_r K$) in $\partial^2H$ (resp. $\partial^2_r H$).

It therefore follows that if $H$ acts geometrically on a hyperbolic metric space $X$ (resp.
$(H, \PP)$ acts in a  strictly type preserving manner on a relatively hyperbolic space $(X, \HH)$) such that
the pair $(H, X)$ has a Cannon-Thurston map, then the pair $(K, X)$ has a Cannon-Thurston map given by 
a composition of the embedding
of $\partial K$ into $\partial H$ (resp. $\partial_r K$ into $\partial_r H$) followed by the Cannon-Thurston map from $\partial H$ to $\partial X$ (resp. $\partial_r H$ into $\partial_r X$).
Further,
$$\lct(K,X) = \lct(H,X) \cap \partial^2K,$$
(resp.
$$\lct(K,X) = \lct(H,X) \cap \partial^2_rK,)$$
where the intersection is taken in $\partial^2H$
(resp.$\partial^2_rH$). 

Since all finitely generated infinite index subgroups $K$ of free groups and surface groups are quasi-convex
(resp. relatively quasi-convex), this
applies, in particular, when $H$ is a free group or a surface group.
\label{subct}
\end{rmk}

\subsection{Algebraic Ending Laminations}\label{ael} In \cite{mitra-endlam},
the first author gave a different,  more group theoretic
description of ending laminations motivated by Thurston's description in \cite{thurstonnotes}.
Thurston's description uses a transverse measure which is eventually forgotten \cite{klarreich-el, br-bdy}, whereas the approach in
\cite{mitra-endlam} uses Hausdorff limits and is purely topological in nature.
We rename the ending laminations of \cite{mitra-endlam}  {\bf algebraic ending laminations} to emphasize the difference.

Thus some of the topological aspects  of Thurston's theory of ending laminations
were generalized  to the context of
normal hyperbolic subgroups of hyperbolic groups and used to give
an explicit  description of the continuous boundary
extension 
$\hat{i} : \widehat{\Gamma_H} \rightarrow \widehat{\Gamma_G}$  occurring in Theorem \ref{ctexist}.

Let $$1 \rightarrow H \rightarrow G \rightarrow Q \rightarrow 1$$
be an exact sequence of finitely presented groups where $H$, $G$
and hence $Q$ (from \cite{mosher-hypextns}) are hyperbolic. In this setup one has   algebraic
ending laminations (defined below)
naturally parametrized by points in the boundary $\partial Q$
of the quotient
group $Q$.

Corresponding to every element $g\in G$ there exists an automorphism 
of $H$ taking $h$ to $g^{-1}hg$ for $h\in H$. Such an automorphism induces
a bijection $\phi_g$ of the vertices of $\Gamma_H$. This gives rise to a map
from $\Gamma_H$ to itself, sending an edge [$a,b$] linearly to a shortest
edge-path joining $\phi_g (a)$ to $\phi_g (b)$. 

Fix $z\in{\partial{Q}}$ and let $[1,z)$ be a  geodesic ray in $\Gamma_Q$ starting at the
identity $1$ and converging to $z\in{\partial}{Q}$. Let
$\sigma$ be a single-valued quasi-isometric (qi) section of $Q$ into $G$.
The existence of such a qi-section $\sigma$ was proved by Mosher \cite{mosher-hypextns}.
Let $z_n$ be the vertex on $[1,z)$ such that ${d_Q}(1,{z_n}) = n$ and let
${g_n} = {\sigma}({z_n})$. 

Next, fix $h\in{H}$. A geodesic segment $[a, b] \subset \Gamma_H$ will be called a 
{\bf free homotopy representative} (or shortest representative in the
same conjugacy class) of $h$, if
\begin{enumerate}
	\item $a^{-1}b$ is conjugate to $h$ {\bf in $H$},
	\item The length of $[a,b]$ is shortest amongst all such conjugates of $h$
	in $H$.
\end{enumerate}
Let ${\LL}_0^h$ be the ($H$--invariant) collection of
all free homotopy representatives of $h$ in $\Gamma_H$.
Intuitively, ${\LL}_0^h$ can be thought of as the collection of all geodesic segments in $\Gamma_H$ that are lifts of shortest closed geodesics in $\Gamma_H/H$ in the same conjugacy class as $h$ (in the setting of a closed manifold of negative curvature, these would be geodesic segments that are path-lifts of the unique closed geodesic in the free homotopy class of a closed loop denoting $h$). Identifying equivalent geodesics (i.e. geodesics sharing the same set of end-points) in $\LL_0^h$ one obtains a subset
$L_0^h$ of (ordered) pairs of points in 
${\widehat{{\Gamma}_H}}$. Next, let ${\LL}_n^h$ be the ($H$--invariant) collection of
all free homotopy representatives of  ${\phi}_{g_n^{-1}}(h) (=g_n h g_n^{-1})$ in $\Gamma_H$. Again, identifying equivalent geodesics  in $\LL_n^h$ one obtains a subset
$L_n^h$ of (ordered) pairs of points in 
${\widehat{{\Gamma}_H}}$. 

See diagram below, where the long vertical arrow on the right depicts the geodesic ray $[1,z)$ in $\Gamma_Q$. We assume that $h$ is chosen to be a free homotopy representative of itself. The corresponding path is assumed to lie in the translate (or alternately, coset) $g_n \Gamma_H$. Then $g_nhg_n^{-1}$ is a path starting and ending in $\Gamma_H$ and we pass to its free homotopy representative in $\Gamma_H$ to get an element of ${\LL}_n^h$. It is important to note that elements of ${\LL}_n^h$ are geodesics in $\Gamma_H$, but not in $\Gamma_G$. What we are intuitively doing here is looking at a closed loop $\sigma_h$ based at $g_n$ in $\Gamma_G/H$ corresponding to $h$ and sitting over $z_n \in [1,z)$. We then concatenate in order
\begin{enumerate}
\item  A path $\sigma_n$ from $1$ to $g_n$. The word in $G$ denoting $\sigma_n$ is $g_n$.
\item This is followed by $\sigma_h$. The word in $G$ denoting $\sigma_h$ is $h$.
\item This is  followed by $\overline{\sigma_n}$ (the "opposite" path to $\sigma_n$). The word in $G$ denoting  $\overline{\sigma_n}$ is $g_n^{-1}$.
\end{enumerate}
This gives a loop based at $1 \in \Gamma_G/H$, and then "homotoping" it back to $\Gamma_H/H$ and "tightening" we get a free homotopy representative.

\begin{center}
	
	\includegraphics[height=6cm]{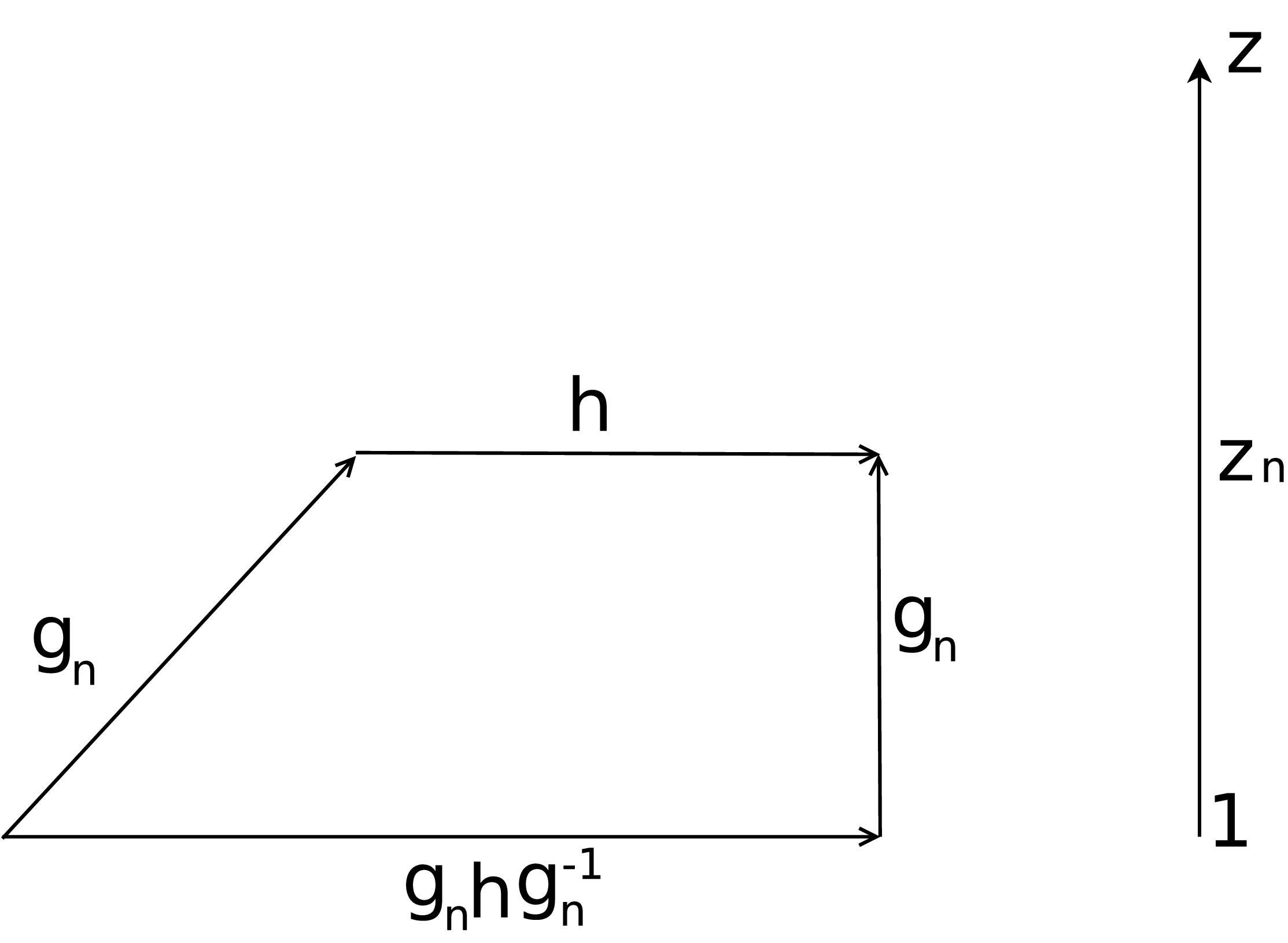}

\end{center}

\begin{defn}\label{lelz}
   
 The intersection 
with ${\partial}^2H$ of the union
of all subsequential limits (in the product topology on ${\widehat{{\Gamma}_H}}\times {\widehat{{\Gamma}_H}}$)
of $\{{L_n^h}\}$ is denoted by 
${\Lambda}_{0z}^h$. It is clear that ${\Lambda}_{0z}^h$ and ${\Lambda}_{0z}^{h^{-1}}$ are related by the flip.

 The  {\bf  algebraic  ending pre-lamination} corresponding to
$z\in{\partial}{\Gamma_Q}$  is given by
$$\lel^{z} = {{\bigcup}_{h\in{H}}}{{\Lambda}_{0z}^h},$$ and the  {\bf  algebraic  ending lamination} corresponding to
$z\in{\partial}{\Gamma_Q}$  is given by the closure $\overline{\lel^{z}}$.
\end{defn} 

We indicate the slight modification to the above definition necessary to make it work for a
hyperbolic metric bundle $X$ over a ray $[0, \infty)$, with fibers  universal covers of (metric)
surfaces or graphs as in Section \ref{sec:mbdl}. One prefers to think of the vertex spaces as corresponding to  integers and edge spaces corresponding to  intervals $[n-1,n]$ where $n \in \natls$.  Let $\sigma : [0, \infty) \to X$ be a qi-section \cite[Proposition 2.10]{mahan-sardar}
through the identity element in the fiber $H_0$ over $0$. The fiber $H_n$ over $n$ is acted upon cocompactly by a (surface or free) group $H$. Further, thickness of the ray guarantees that the quotient of each fiber by $H$ is of {\it uniformly} bounded diameter. Each $H_n$ contains a preferred set of points (vertices) given by the  $H-$orbit of $\sigma (n)$.  For  $n \in \natls$ and $x \in H\sigma(n)$, there exists a  unique $H-$translate $\sigma_x$ of $\sigma ([0, \infty))$ through $x$. Since $H_n/H$ is of uniformly bounded diameter (independent of $n$) it makes sense  to consider the ($H$--invariant) collection ${\LL}_n^D$   of all free homotopy
 representatives of  $ \sigma ([0,n]) [x_n,y_n] \overline{h_n\sigma ([0,n])}$,
  where $x_n = \sigma(n),  y_n\in H\sigma(n)$, $d_X (x_n,y_n) \leq D$ and $\overline{h_n\sigma([0,n])}$ 
  denotes the translate of $\sigma ([0,n])$ through $y_n$ with reverse orientation.  
  	As before, this gives a subset	$L_n^D$ of (ordered) pairs of points in $\widehat{{H_0}}$. The intersection 	with
  		 ${\partial}^2H_0$ of the union of all subsequential limits of $\{{L_n^D}\}$ is denoted by ${\Lambda}^D$. Note that ${\Lambda}^D$ is invariant under the flip here. 

\begin{defn}\label{lelz-mb}
The  {\bf  algebraic  ending pre-lamination} corresponding to
$z  = r(\infty)$  is given by
$$\lel^z = {{\bigcup}_{D\in \natls}{\Lambda}^D},$$ and the  {\bf  algebraic  ending lamination} is given by the closure $\overline{\lel^{z}}$. We denote $$\lel= \overline{\lel^{z}}.$$
	(Here the superscript $z$ is initially used for the sake of consistency with the notation in Definition \ref{lelz} and then dropped to be consistent with Definition \ref{lel} below.)
\end{defn}

 In  Definition \ref{lelz} above we have followed \cite{mitra-endlam}.
As was pointed to us by the referee, the fact that we are choosing free homotopy representatives and shortest representatives implies that we are in fact applying  ${\phi}_{g_n^{-1}}$ to the conjugacy class $[h]$ rather than $h$ itself. However, once we have applied ${\phi}_{g_n^{-1}}$ to $[h]$, we need to choose shortest representatives and their cyclic permutations in order to extract  subsequential limits. We have made the choice here so that we can quote Theorem \ref{ct=el} below directly from \cite{mitra-endlam}. Further, the generalization to Definition \ref{lelz-mb} becomes natural with this choice.

Note also that ${\Lambda}_{0z}^h$ and $\Lambda^D$ are indeed closed as we are taking all subsequential limits. However, closedness may be destroyed when we take the union over all $h$ (or $D$); hence the term pre-lamination. By Theorem \ref{ct=el} below, $\lel^{z}$ is actually a lamination in the context we are interested in.

We explain  the convention of using ${\phi}_{g_n^{-1}}$ in the motivating case of the cover of a hyperbolic 3-manifold fibering over the circle \cite{thurstonnotes} corresponding to the
 fiber $S$. The group $Q$ is $\Z$ here and the fiber over $n$ is denoted $S \times \{ n \}$.
 Here $h$ should be thought of as (a lift of) a bounded length curve $\sigma$ on $S\times \{ n \}$.
Also ${\phi}_{g_n^{-1}}(h)$ may be thought of in this case as (a lift of)  the closed geodesic on 
$S\times \{ 0 \}$
freely homotopic to $\sigma$. The ending lamination in this situation is obtained by taking limits of such closed geodesics
in a suitable topology (which is not important for us here).

\begin{defn}\label{lel} The  
 {\it algebraic   ending lamination} $\lel$ for the triple $(H,G,Q)$ is defined
by 
$$\lel = \lel (H,G,Q)  = {{\bigcup}_{z\in{\partial}{\Gamma_Q}}}\overline{\lel^{z}}.$$
\end{defn}
It follows from \cite{mitra-endlam} that $\lel$ is in fact closed and hence an algebraic lamination in our sense.
The main theorem of \cite{mitra-endlam} equates $\lel$ and $\lct$.

\begin{theorem} \cite{mitra-endlam} \label{ct=el} $\lel (H,G,Q) = \lel = \lct = \lct (H,G)$. \end{theorem}

We shall be needing a slightly modified version of Theorem \ref{ct=el} later, when we consider hyperbolic metric bundles $X$ over rays $[0, \infty)$, with fibers  universal covers of (metric)
surfaces or graphs as in Section \ref{sec:mbdl}. We note here
that the proof in \cite{mitra-endlam} goes through in this case, too, with small modifications. 
We outline the steps of the proof in \cite{mitra-endlam} here and indicate the technical modifications from
\cite{mahan-sardar}.

\begin{prop} \label{ct=elrmk} Let $X$ be a hyperbolic metric bundle  over a ray $[0, \infty)$, with fibers  universal covers of (metric)
	surfaces or graphs as in Section \ref{sec:mbdl} and let $H$ be the associated surface or free group. Let $\lel$ denote the algebraic ending lamination from Definition \ref{lelz-mb}. Then $ \lel = \lct = \lct (H,X)$.
\end{prop}

\begin{proof} {\bf (Sketch of steps following \cite{mitra-endlam})}
	
The proof of Lemma 3.5 in \cite{mitra-endlam} goes through directly establishing that
$\lel \subset \lct$.

The crucial technical tool in \cite{mitra-endlam} after this is the construction of a ladder. The 
corresponding construct in the metric bundle context is given in Section 2.2  of \cite{mahan-sardar}
and generalizes the construction in \cite{mitra-ct, mitra-endlam}. 
Quasi-convexity of ladders when the metric bundle is hyperbolic is now established by Theorem 3.2 of
\cite{mahan-sardar}.

The proof of aperiodicity of ending
laminations established in Section 4.1 of \cite{mitra-endlam}
uses only the group structure of the fiber (but not of the total space)
and hence goes through with $\sigma ([0,n])$
replacing the quasi-geodesic $[1,g_n]$. 

The final ingredient in the proof is the fact that qi-sections coarsely separate ladders (Lemma 4.8
in Section 4.2 of 
\cite{mitra-endlam}). The proof is the same in the case of metric bundles.

With all these ingredients in place, the proof of Theorem 4.11 of \cite{mitra-endlam} now goes through
in the more general context of metric bundles to establish that $\lel = \lct$.
\end{proof}

\subsubsection{Surface Ending Laminations} It is appropriate to explicate at this juncture the relation between the ending laminations introduced by Thurston
in \cite[Chapter 9]{thurstonnotes}, which we call {\bf surface ending laminations } henceforth, and the algebraic ending laminations we have been discussing.  This will be particularly relevant when we deal with surface Kleinian groups,
where the surface has punctures.
Work of several authors including \cite{minsky-jams, klarreich-el, bowditch-ct, mahan-elct}
explore related themes.

The Thurston boundary
$\partial Teich(S)$ consists of projectivized measured laminations on $S$.
 Let $r$ be a thick hyperbolic geodesic ray in \Teich space $Teich(S)$ where $S$ is a surface possibly with punctures. Then, by a result of Masur \cite{masur-ue}, it has a unique ideal point $r_\infty  \in \partial Teich(S)$ corresponding to a uniquely ergodic lamination.  Let $\lel (r_\infty)$ be the geodesic lamination underlying 
$r_\infty $.  Let $X_0$ be the  universal curve over $r$. Let $X_1$ denote $X_0$  with a small neighborhood of the cusps removed. 
 Minsky proves \cite{minsky-jams} that $X_1$ is (uniformly) bi-Lipschitz homeomorphic to the convex
core minus (a small neighborhood of) cusps of the unique simply degenerate hyperbolic 3-manifold $M$ with conformal structure on the geometrically finite end given by $r(0) \in Teich(S)$
and ending lamination of the simply degenerate end given by  $\lel (r_\infty)$. The convex core of $M$ is denoted by $Y_0$ and 
 let $Y_1$ denote $Y_0$  with a small neighborhood of the cusps removed. Thus $X_1, Y_1$ are  (uniformly) bi-Lipschitz homeomorphic.
Let $X$ denote the universal cover of $X_1$ and $\HH$ its collection of boundary horospheres. Then $X$ is (strongly) hyperbolic relative to $\HH$.
 Let $H  =\pi_1(S)$ regarded as a relatively hyperbolic group, hyperbolic rel. cusp subgroups. The relative hyperbolic (or Bowditch) boundary $\partial_r H$
of the relatively hyperbolic group
is still the circle (as when $S$ is closed) and $\partial^2_r H$ is defined as $(\partial_r H \times \partial_r H \setminus \Delta)$ as usual.
The existence of a Cannon-Thurston map in this setting of a relatively hyperbolic group $H$ acting on a relatively hyperbolic space $(X, \HH)$
has been proven in \cite{bowditch-ct} (see also \cite{brahma-bddgeo}).

The {\bf  diagonal closure} $Diag(\LL)$ of a surface lamination $\LL$ is an algebraic lamination given by the 
 transitive closure of the relation
defined by $\LL$ on $\partial^2 H$.
  The {\bf closed diagonal closure} $\LL^d$ of a surface lamination $\LL$ is an algebraic lamination given by the 
  {\em closure in $\partial^2 H$ of} the transitive closure of the relation
defined by $\LL$ on $\partial^2 H$.  When $S$ is closed, each complementary ideal polygon of
 $\LL$ has finitely
many sides; so the closed diagonal closure $\LL^d$ 
agrees with the  diagonal closure $Diag(\LL)$
and comprises the original lamination $\LL$ along with the union
of these diagonals (which are allowed to intersect). For a punctured surface $S^h$ however, it is not enough just to take
the transitive closure of the relation
defined by $\LL$. In this case, the fundamental group $H$ is free and equals that of a compact core $S^K$ of $S^h$
(i.e. a compact submanifold of $S^h$ whose inclusion induces a homotopy equivalence). The lamination 
thought of as a subset of $\til{S^K}$, now has a complementary domain with infinitely many (bi-infinite) sides
(the so-called "crown domain") one of which corresponds to a lift $\til \sigma$ of a boundary component
$\sigma$ of $S^K$. The transitive closure of $\LL$ {\em does not} include the boundary points of $\til \sigma$
in particular. However, the closure (in $\partial^2 H$) 
of the  transitive closure of $\LL$ captures all these, and is also closed
under the transitive closure operation. We shall return to this later when dealing with punctured surfaces.

\begin{theorem}
 \cite{minsky-jams, bowditch-ct} Let $r$ be a thick hyperbolic geodesic in $Teich(S)$ and  let $\lel (r_\infty)$ denote its end-point
in $\partial Teich(S)$ regarded as a surface lamination. 
 Let $X$ be the universal cover of $X_1$.
Then $\lct (H, \til{M})=\lct (H,(X, \HH)) =  \lel (r_\infty)^d$.
\label{end1} \end{theorem}

Note that Theorem \ref{end1} holds both for closed surfaces as well as surfaces with finitely many punctures.

\subsubsection{Generalized algebraic ending laminations} The setup of a normal hyperbolic subgroup of a hyperbolic subgroup is quite restrictive.
Instead we could consider $H$ acting geometrically on a hyperbolic metric space $X$. Let $Y = X/H$ denote the quotient.
Let $\{ \sigma_n \}$ denote a sequence of free homotopy classes of closed loops in $Y$ (these necessarily correspond to conjugacy classes in $H$) such that
the geodesic realizations  of  $\{ \sigma_n \}$  in $Y$ exit all compact sets. Then subsequential limits of all such sequences define again
an algebraic lamination, which we call a {\bf generalized algebraic ending lamination} and denote $\lgel (=\lgel (H,X))$.

Then Lemma 3.5 of \cite{mitra-endlam} (or Proposition 3.1 of \cite{mahan-elct} or Section 4.1 of \cite{mahan-kl}) gives 

\begin{prop} If the pair $(H, X)$ has a Cannon-Thurston map, then $$\lgel (H, X) \subset \lct (H, X).$$ \label{ctgel}\end{prop}

\subsection{Laminations dual to an $\R-$tree} We recall 
some of the material from \cite[Section 3.1]{bestvina-rtree}  on convergence of a sequence $\{ (X_i, *_i, \rho_i \}$
of based $H-$spaces for $H$ a fixed group.

An {\bf $H$-space} is a pair
$(X,\rho)$ where $X$ is a metric space and $\rho:H\to Isom(X)$ is a
homomorphism. Equivalently, it is an {\bf action} of $H$ on $X$ by isometries.
Let $d_X$ denote the metric on $X$.
 A triple $(X,*,\rho)$ (for $* \in X$) is a {\bf based $H-$space} if
$(X,\rho)$ is an $H$-space and $*$, also called the base-point, is not a global 
fixed point under the action of $H$. 

The space of all non-zero pseudo-metrics (or distance functions)  on $H$, equipped with the 
compact-open topology is denoted by $\DD$ (the condition that $*$ is not a global 
fixed point guarantees that $\DD$ is non-empty). 
Note that an element of $\DD$ is a non-negative real valued function on $H \times H$.
Assume that $H$ acts on $H\times H$
diagonally, and on $[0,\infty)$ trivially. Let 
${\EE\DD}\subset \DD$ denote the subspace of $H$-equivariant pseudo-metrics
under this action. Projectivizing 
${\EE\DD}$ (using the scaling action and passing to the quotient), 
 we obtain the {\bf projectivized
 	equivariant distance functions} denoted by $\mathcal
PED$. A pseudo-metric on $H$ is said to be
$\delta$-hyperbolic if the associated metric space is
$\delta$-hyperbolic (the equivalence class of the identity element is taken to be
the base-point). 

A based $H$-space $(X,*,\rho)$ induces an equivariant pseudo-metric
$d=d_{(X,*,\rho)}$ on
$H$ by defining
$d(g,h):=d_X(\rho(g)(*),\rho(h)(*))$. 
If the stabilizer of $*$ under the induced action  is trivial, then $H$ can, as usual, be
identified with the orbit of $*$. This gives an induced metric
 $d_{(X,*,\rho)}$ on $H$.

\begin{defn} \cite{bestvina-rtree}
A sequence $(X_i,*_i,\rho_i)$, $i=1,2,\cdots$ of
based $H$-spaces {\bf converges} to the based $H$-space $(X,*,\rho)$
if
 $[d_{(X_i,*_i,\rho_i)}]\to [d_{(X,*,\rho)}]$ in $\mathcal PED$. We denote this 
 as
 $\lim_{i\to\infty}(X_i,*_i,\rho_i)=(X,*,\rho)$.
\end{defn}

\begin{theorem} \label{cgnce}\cite[Theorem 3.3]{bestvina-rtree}
Let $(X_i,*_i,\rho_i)$ be a convergent sequence of based $H$-spaces such that
\begin{enumerate}
\item there exists $\delta\geq 0$ such that each $X_i$ is $\delta$ hyperbolic, 
\item there exists $h\in H$ such that the sequence
$d_i=d_{X_i}(*_i,\rho_i(h)(*))$ is unbounded.
\end{enumerate}
Then there is a based $H$-tree $(T,*)$ (without global fixed points) and an isometric action
$\rho:H\to Isom(T)$ such that 
 $(X_i,*_i,\rho_i)\to (T,*,\rho)$.
\end{theorem}

 Note that convergence of $(X_i,*_i,\rho_i)$ (in terms of projective length functions in $\mathcal PED$) forces uniqueness of the projectivized length function. In particular, if there is an $h'$ such that the growth rates $d_i'=d_{X_i}(*_i,\rho_i(h')(*))$ are much greater than $d_i$  (more than linear), there would not be an action of $H$ on the limit space as $h'$ would be forced to translate $*$ by an infinite distance {\it after projectivizing}. Thus implicitly, the hypothesis of Theorem \ref{cgnce} selects out the maximal growth rate of the $d_i$'s and scales by these.

\begin{defn} For a convergent sequence $(X_i,*_i,\rho_i)$ as in Theorem \ref{cgnce} above we define a dual  algebraic lamination as follows:\\
 Let $h_i$ be any sequence 
such that  
$$\frac{d_{(X_i,*_i,\rho_i)}(1,h_i)}{d_i} \to 0.$$
The collection of all limits of $(h_i^{-\infty}, h_i^{\infty})$ in $\partial^2H$
will be called the dual ending lamination corresponding to
the sequence  $(X_i,*_i,\rho_i)$ and will be denoted by $\lre \{ (X_i,*_i,\rho_i)\}$.\end{defn}

Next, let $1 \rightarrow H \rightarrow G \rightarrow Q \rightarrow 1$ be an exact sequence of hyperbolic groups.
As in Section \ref{ael}  let $z\in{\partial{Q}}$ and let $[1,z)$ be a  geodesic ray in $\Gamma_Q$; let
$\sigma$ be a single-valued quasi-isometric section of $Q$ into $G$.
Let $z_i$ be the vertex on $[1,z)$ such that ${d_Q}(1,{z_i}) = i$ and let
${g_i} = {\sigma}({z_i})$. Now, let $X_i = \Gamma_H$, $*_i = 1 \in \Gamma_H$ and let $\rho_i (h) (*) =  {\phi}_{g_i^{-1}} (h)  (*)$.
With this notation  the following Proposition is immediate from Definition \ref{lelz}:

\begin{prop}\label{lregel0}
$\lel^z \subset \lre \{(X_i,*_i,\rho_i)\}$.\end{prop}

An alternative description can be given directly in terms of the action on the limiting $\R-$tree in Theorem \ref{cgnce}
as follows.
The ray $[1,z) \subset Q$ defines a graph $X_z$ of spaces where the underlying graph is a ray $[0, \infty)$ with vertices at the integer points
and edges of the form $[n-1,n]$. All vertex and edge spaces are abstractly isometric to $\Gamma_H$. Let $e_n  = {g_{n-1}}^{-1}g_n$.
The edge-space to vertex space
inclusions are given by the identity to the left  vertex space and by $ {\phi}_{e_n}$ to the right. We call $X_z$ the {\bf universal metric bundle}
over $[1,z)$ (though it depends on the qi section $\sigma$ of $Q$ used as well).  Hyperbolicity of $X_z$ is equivalent to the {\it flaring} condition
of Bestvina-Feighn \cite{BF} as shown for instance in \cite{mahan-sardar} in the general context of metric bundles.

Suppose now that the sequence  $\{(X_i,*_i,\rho_i)\}$
 with  $X_i = \Gamma_H, *_i = 1 \in \Gamma_H, \rho_i (h) (*) =  {\phi}_{g_i^{-1}} (h)  (*)$
converges as a sequence  of $H-$spaces to an $H-$action on an $\R-$tree $T = T(\{ X_i,*_i,\rho_i\})$. Generalizing the construction of Coulbois, Hilion and Lustig
 \cite{chl08a, chl08b} to the hyperbolic group $H$ we have the following notion of an algebraic lamination (contained in $\partial^2 H$) dual to $T$. The translation
length in $T$ will be denoted as $l_T$.

\begin{defn}\label{def-lre}
 Let
$$
L_{\epsilon}(T)=\overline{\big\{(g^{-\infty},g^{\infty}) : l_T(g)<\epsilon \big\}}
$$
where $\overline{A}$ denotes the closure of $A$. Define
$$
\lre\big\{(X_i,*_i,\rho_i) \big\}= \lre(T)= \cap_{\epsilon > 0} L_{\epsilon} (T).
$$
\end{defn}

\section{Closed Surfaces}\label{surf} 

\subsection{Arationality}
Establishing arationality of $\lct$ for surface laminations arising out of a thick hyperbolic 
ray or an exact sequence of hyperbolic groups  really involves identifying the  algebraic 
Cannon-Thurston lamination $\lct$ with (the original) geodesic laminations introduced by 
Thurston \cite{thurstonnotes}. To distinguish them from algebraic laminations, we shall 
refer to geodesic laminations on surfaces as surface laminations.  The results of this 
subsection (though not the next subsection) hold equally for $S$ compact or finite volume non-compact.

A surface lamination  $\LL \subset S$ is {\bf arational}  if it has no closed leaves. 
It is called {\bf  filling} if it  intersects every essential non-peripheral closed curve on the surface and {\bf minimal} if it equals the closure of any of its leaves. Note that for an arational minimal lamination, the complement consists of ideal polygons. Adjoining some (non-intersecting) diagonals, we can still obtain an arational lamination, which is however no longer minimal. However, from an arational lamination we can obtain a unique arational minimal lamination by throwing away such diagonal leaves.
Being filling is equivalent to saying that all complementary components of $\LL$
are either topological disks or once punctured disks. Note that every  filling lamination is 
automatically arational. We say that a bi-infinite geodesic $l$ in $\til S$ is {\bf carried by a 
subgroup $K \subset H(=\pi_1(S))$} if both end-points of $l$ lie in the limit set 
$\Lambda_K \subset \partial {\til S}$. A  surface lamination $\LL \subset S$ is 
{\bf strongly arational} if no leaf of $\LL$  or a diagonal in a complementary ideal 
polygon is carried by a finitely generated infinite index subgroup $K$ of  $H$.
The next Lemma holds for both compact and non-compact hyperbolic surfaces of finite volume.
\begin{lemma} Any minimal arational geodesic lamination $\LL_0$ on a finite volume complete hyperbolic surface 
$S$ is strongly arational. \label{aratst} \end{lemma}

\begin{proof} We assume that $S$ is equipped  with a complete finite volume hyperbolic metric and  suppose that $\LL_0$ is a minimal arational geodesic lamination. 
Consider  a finitely generated infinite index subgroup $K$ of  $H$.
By the LERF property of surface groups \cite{scott-lerf}, there exists a finite-sheeted cover 
$S_1$ of $S$ such that $K$ is a geometric subgroup of $\pi_1(S_1)$,
i.e.\ it is the fundamental group of an embedded incompressible subsurface $\Sigma$ 
of $S_1$ with geodesic boundary $\alpha_1$. Let $\LL_1$ be the lift of $\LL_0$ to $S_1$.

We now show that $\LL_1$ is minimal arational. Since $\LL_0$ is arational, and leaves of $\LL_1$ are lifts of leaves of $\LL_0$, arationality of $\LL_1$ follows. Let $l$ be any leaf of $\LL_1$ and $\overline{l}$ be the closure of $l$ in $\LL_1$. Note that there are no diagonal leaves in $\LL_1$ as such a leaf would have to come from a diagonal leaf in $\LL_0$. If $\overline{l} \neq \LL_1$, then $\LL_1$ is not minimal and  must contain a closed leaf $l'$. Since we have already shown that $\LL_1$ is arational, this is a contradiction.
Hence, all leaves of 
$\LL_1$ are dense in $\LL_1$, i.e.\ $\LL_1$ is minimal as well. 

In fact, any diagonal in a complementary ideal polygon of $\LL_1$
is forward asymptotic to a leaf of  $\LL_1$ and
is therefore also dense in $\LL_1$;  in particular it intersects $\alpha_1$. 
Hence no leaf of $\LL_1$, nor a diagonal in a complementary ideal polygon,  
is carried by $\Sigma$. The result follows. \end{proof}

\begin{theorem} \cite{klarreich-el} The boundary  $\partial CC(S)$  of the curve complex $CC(S)$ consists of   minimal arational geodesic surface laminations.\label{bdycc} \end{theorem}

The following Theorem may be taken as a definition of convex cocompactness for subgroups of the mapping class group of a  surface with (at most) finitely many punctures.

\begin{theorem}\cite{farb-mosher, kl,ham-cc}  A subgroup $Q$ of $MCG(S)$ is convex cocompact if and only if some (any) orbit of $Q$ in the curve complex
$CC(S)$ is qi-embedded. \label{coco-surf} \end{theorem}

Recall that the Thurston boundary $\partial Teich(S)$  of  \Teich space is the space of projectivized measured laminations.
The following Theorem gives us the required strong arationality result.

\begin{theorem} \label{mix-surf}
	Let $S$ be a complete hyperbolic surface of finite volume.
Let $r$ be a thick hyperbolic ray in \Teich space $Teich(S)$ and let $r_\infty \in \partial Teich(S)$ be the limiting surface lamination.
Then $r_\infty$ is strongly arational.

In particular if $Q$ is a convex cocompact  subgroup of $MCG(S)$ and $r$ is a quasi-geodesic ray in $Q$ starting at $1 \in Q$, then 
its limit $r_\infty$ in the boundary $\partial CC(S)$ of the curve complex is strongly arational.
\end{theorem}

\begin{proof}  Recall that for a thick hyperbolic ray  $r$ in $Teich(S)$,  
$r_\infty \in \partial CC(S)$.  For the second statement of the Theorem,  $\partial Q $ embeds as a subset of  $\partial CC(S)$ by Theorem \ref{coco-surf}
and hence the boundary point $r_\infty \in \partial CC(S)$ as well.

By Theorem \ref{bdycc}, $r_\infty$ is an arational minimal lamination.
Hence by Lemma \ref{aratst}, $r_\infty$ is strongly arational.\end{proof}

\subsection{Quasi-convexity}
We now turn to closed surfaces. Let 
$$
1 \to H \to G \to Q \to 1
$$ 
be an exact sequence of hyperbolic groups with $H=\pi_1(S)$ for a closed hyperbolic surface
$S$. Then $Q$ is convex cocompact \cite{farb-mosher} and its orbit in both $Teich(S)$ and $CC(S)$ are quasi-convex. By Theorem \ref{ct=el}
 $\lel (H,G,Q) = \lel = \lct = \lct (H,G)$. Further $\lel = \cup_{z \in \partial Q} \overline{\lel^z}$. Recall that
$\overline{\lel^z}$ denotes the algebraic ending lamination corresponding to $z$ and $\lel (z)$ denotes the surface ending lamination corresponding to $z$.
By Theorem \ref{end1}, $\overline{\lel^z} = \lel(z)^d$. We combine all this as follows.

\begin{theorem} \cite{minsky-jams, mitra-endlam} \label{end2} 
If
$$
1 \to H \to G \to Q \to 1
$$ 
is an exact sequence with $Q$  convex cocompact and $H=\pi_1(S)$ for a closed surface $S$ of genus greater than one, and 
$$
z=r_\infty\in \partial Q \subset   \partial CC(S)
$$ 
then any lift of $[1,z)$
to $Teich(S)$ is thick hyperbolic. Further, 
$$
\lct (H,G) = \cup_{z \in \partial Q} \lel (z)^d.
$$ 
\end{theorem}

We are now in a position to prove the main Theorems of this Section.

\begin{theorem} \label{surf-ray}  Let $H=\pi_1(S)$ for $S$ a closed surface of genus greater  than one.
Let $r$ be a thick hyperbolic ray in \Teich space $Teich(S)$ and let $r_\infty \in \partial Teich(S)$ be the limiting surface ending lamination.
Let $X$ denote the universal metric bundle over $r$. Let $K$ be a finitely generated infinite index subgroup of $H$. Then any orbit of $K$ in $X$  is quasi-convex.
\end{theorem}

\begin{proof} By Theorem \ref{mix-surf}, the lamination $r_\infty$  is strongly arational. Hence no leaf or diagonal of $r_\infty$ is carried by $K$.
By Theorem \ref{end1}, the Cannon-Thurston lamination $\lct (H, X)= \lel (r_\infty)^d$. 
 Hence no leaf of $\lct (H,X)$ is carried by $K$. By Lemma \ref{qccrit} and Remark \ref{subct}, any orbit of $K$ in $X$
is quasi-convex in $X$.  \end{proof}

The next Theorem was proven by Dowdall, Kent and Leininger \cite[Theorem 1.3]{dkl} by different methods.
\begin{theorem} \label{surf-coco}
Let 
$$
1 \to H \to G \to Q \to 1
$$
be an exact sequence of hyperbolic groups with $H=\pi_1(S)$ ($S$ closed) and $Q$ convex cocompact. 
Let $K$ be a finitely generated infinite index subgroup of $H$. Then  $K$ is quasi-convex  in $G$.
\end{theorem}

\begin{proof}  As in the proof of Theorem \ref{surf-ray} above the lamination ${\lel} (z)$  is strongly arational for each $z \in \partial Q \subset \partial CC(S)$
(where we identify the boundary of $Q$ with the boundary of its orbit in $CC(S)$). Hence for all $z  \in \partial Q$, no leaf of $\lel (z)^d$ is carried by $K$.
By Theorem \ref{end2}, 
$$
\lct (H,G) = \lel (H,G)= \cup_{z \in \partial Q} \lel (z)^d.
$$ 
Hence no leaf of $\lct (H,G)$ is carried by $K$. By Lemma \ref{qccrit} and Remark \ref{subct},  $K$ 
is quasi-convex in $G$.  \end{proof}

\section{Free Groups}\label{free}

 For the purposes of this section, $H = F_n$ is free. 

\subsection{Arationality} Recall that the (unprojectivized) Culler-Vogtmann Outer
space corresponding to $F_n$ is denoted by $cv_n$ and its boundary
by $\partial cv_n$. The points of $\partial cv_n$ correspond to  very small actions of $F_n$ on $\R-$trees.

\begin{defn}\cite{gui}
An $\R$-tree $T\in  \partial cv_n$ is said to be {\bf indecomposable} if for any non-degenerate segments $I$ and
  $J$ contained  in $T$, there exist finitely many elements $g_1,\cdots,g_n \in F_n$ such that
\begin{enumerate}
\item $I\subset \bigcup_{i  = 1 \cdots n} g_iJ$
\item $g_iJ\cap g_{i+1}J$ is a non degenerate segment for any $i=1,  \cdots  ,  n-1$.
\end{enumerate}
\end{defn}

Dual to $T \in \partial cv_n$ is an algebraic lamination
$\lre(T)$ defined as follows (which we had generalized to Definition \ref{def-lre} for general hyperbolic groups):

\begin{defn} \label{lre} \cite{chl08a, chl08b}
Let
$
L_{\epsilon}(T)=\overline{\big\{(g^{-\infty},g^{\infty})|l_T(g)<\epsilon \big\}}
$. Define
$\lre(T):= \cap_{\epsilon > 0} L_{\epsilon} (T)$.
\end{defn}

\begin{defn} \cite{br-bdy}
	A leaf  $(p,q)$ of an algebraic lamination $L$ is carried by a subgroup $K$ of $F_n$ if both $p$ and $q$ lie in the limit set of $K$.
\end{defn}

\begin{defn}
	A lamination $L$ is called arational (resp. strongly arational) if no leaf of 
	$L$ is carried by a proper free factor of $F_n$
	(resp. by a proper finitely generated infinite index subgroup of $F_n$).

	A tree $T \in \partial cv_n$ is called
	arational  (resp. strongly arational) if $\lre (T)$ is arational  (resp. strongly arational). 
\end{defn}

The free factor complex $\FF_n$ for $F_n$ is a simplicial complex whose vertices are conjugacy classes $A$ of free factors and simplices are chains $A_1 \subsetneq \cdots \subsetneq A_k$ of free factors. 
\begin{defn}\cite{dt1, ham-coco} A subgroup
	$Q$ of  $Out(F_n)$  is said to be convex cocompact in $Out(F_n)$   if some (and hence any)
	orbit of $Q$ in $\FF_n$ is qi embedded.
	
	A subgroup $Q$ of  $Out(F_n)$  is said to be
	purely atoroidal if every element of $Q$ is hyperbolic.\end{defn}

A geodesic or quasi-geodesic (with respect to the Lipschitz metric)
ray $[1,z)$ in  Outer Space $cv_n$ defines a metric bundle $X_z$  where the underlying graph is a ray $[0, \infty)$ with vertices at the integer points
and edges of the form $[n-1,n]$. 
As mentioned in Section \ref{sec:mbdl}, after  moving the initial point of $[1,z)$ by a uniformly bounded
amount, we can assume without loss of generality that $[1,z)$ is a folding path. Further, the $\R-$tree $T_z$ corresponding
to $z$, equipped with an $F_n$ action is exactly the tree encoded by $z \in \partial cv_n$ (tautologically).

We refer the reader to \cite[Section 2.4]{bf2} for details on folding paths and geodesics in Outer Space  
The material relevant to this paper is efficiently summarized in \cite[Section 2.7]{dt1}. We call $X_z$ the {\bf universal metric bundle}
over $[1,z)$. We shall be interested in two cases:
\begin{enumerate}
\item[(Case 1:)]  $[1,z)$ is contained in a convex cocompact subgroup $Q$ of  $Out(F_n)$ and, for $\sigma$ a qi section \cite{mosher-hypextns}, $\sigma ([1,z))$ is identified with the corresponding 
quasi-geodesic ray contained in an orbit of  $Q$ in $cv_n$. The universal metric bundle will (in this case) be considered over $\sigma ([1,z))$.
\item[(Case 2:)]   $[1,z)$ is a thick geodesic ray in $cv_n$, i.e. a geodesic ray projecting to a parametrized quasi-geodesic in the free factor complex $\FF_n$. As mentioned in the Introduction, $[1,z)$ is said to be {\bf thick hyperbolic} if,
in addition, $X_z$ is hyperbolic.
\end{enumerate}

\begin{rmk}
Case 1 above  is directly relevant to Theorem \ref{free-coco}, while Case 2 pertains to Theorem \ref{free-ray}.
These two cases are logically independent, though the proofs are very similar.
\end{rmk}

 The setup used in 
Proposition \ref{lregel} below is extracted from the proof of Theorem 5.2 of \cite{kdt}. We recall the setup of \cite{kdt}.
Assume first that we are in Case 1. Since $Q$ is convex cocompact, we may identify $Q$ with an orbit in $\FF_n$.
 This identification gives  a $Q$-equivariant  embedding of $\partial Q$ into $\partial \FF_n$. We identify
 $\partial Q$ with its image under this embedding. Let $\mathcal{AT}$ consist of the projective classes
 of arational trees in $\partial cv_n$.
 The authors of \cite{kdt} give a natural map (following Bestvina-Reynolds \cite{br-bdy})
 $\partial \pi :\partial \mathcal{AT} \to \partial \FF_n$  associating
 to each arational tree of $\partial cv_n$ the corresponding point in the boundary of the free factor complex
 $\FF_n$.
 Hence, by the identification of $\partial Q$ with its image under the embedding into $\partial \FF_n$,
  each point $z\in\partial Q$ corresponds to
  an equivalence class $T_z$ of arational trees, where two such trees $T_1$
  and $T_2$ are declared equivalent if their associated
  dual laminations $\lre(T_1)$ and $\lre(T_2)$  are the same.

\begin{theorem} \label{kdtthm} \cite[Theorem 5.2]{kdt}
	For each $z \in \partial Q$, there exists $T_z \in \partial cv_n$
	 which is free and arational such that $z \to \partial \pi(T_z)$ under 
	 the embedding of $\partial Q$ into $ \partial \FF_n$ with the property that
	$\overline{\lel^z} = \lre(T_z)$.
\end{theorem}

A remark on a possible ambiguity that might arise from Theorem \ref{kdtthm} as stated above is that $\lre(T_z)$
depends on a choice of a tree lying in the fiber of $\partial \pi : \mathcal{AT} \to \partial \FF_n$. However,
as shown in \cite{br-bdy} (see Theorem \ref{bdy} below), the fiber consists of precisely the elements
of the equivalence class mentioned above and hence $\lre(T_z)$ is well-defined independent of the choice.

We now turn to Case 2 and indicate briefly how the arguments of \cite{kdt} go through in this case
to prove the analogous statement Proposition \ref{lregel} below.
Theorem 4.1 and Lemma 4.12 of \cite{dt1} establish stability of $\FF_n-$progressing quasi-geodesics. 
While  Lemma 5.5 of \cite{kdt} is necessary to prove flaring for Case 1 above, flaring
in Case 2 follows from hyperbolicity. (In fact
it is shown in \cite[Section 5.3, Proposition 5.8]{mahan-sardar}
that flaring is equivalent to   hyperbolicity of $X_z$).
Also, 
 thickness of the ray is by definition for Case 2.
The crucial ingredient for  Theorem 5.2 of \cite{kdt} is Proposition 5.8 of \cite{kdt} which, once Propositions 5.5 and 5.6 of \cite{kdt} are in place, makes no further use of
the fact that $\sigma ([1,z))$ comes from a ray in a convex cocompact $Q$ (Case 1) but just 
that it is thick, stable and that the universal bundle over it satisfies flaring.
Proposition \ref{lregel} as  stated below, now  follows. Note that this part of the argument has nothing to do with
identifying $\lel$ with $\lct$ (the latter is the content of Theorem \ref{ct=el} and Proposition \ref{ct=elrmk}).

\begin{prop}\label{lregel}  Let $[1,z)$ be as in Case 2 above and suppose that  the  universal metric bundle  $X_z$ is hyperbolic. Then
$\lel = \lre (T_z)$.\end{prop}

\begin{rmk}\label{lregelrmk}
A continuously parametrized version of the metric bundle described in Case 2 above occurs in our context of folding paths in Culler-Vogtmann Outer space
$cv_n$ converging to a point  $z \in \partial cv_n$.   The same proof  furnishes $\lel = \lre (T_z)$ in the case of a 
continuously parametrized version of the metric bundle.
\end{rmk}

We collect together a number of Theorems establishing mixing properties for $F_n-$trees.

\begin{theorem}\cite{preynolds-gd}  \label{aratstrong}
 If $T$ is a free indecomposable very small $F_n-$tree then no leaf of the dual lamination $ \lre (T)$
 is carried by a finitely generated subgroup of infinite index in $F_n$. \end{theorem}

\begin{theorem}\cite{preynolds-red}  \label{aratimpliesindec} Let $ T \in \partial cv_n$.
Then $T$ is arational if and only if either\\
a) T is free and indecomposable\\
b) or $T$ is dual to an arational measured foliation on a compact surface $S$ with one boundary component and with
$\pi_1(S)=F_n$.\end{theorem}

 Recall that $\mathcal{AT} \subset
\partial cv_n$ denotes the set of arational trees, equipped with the
subspace topology.  Define a relation $\sim$ on
 $\mathcal{AT}$ by $S\sim T$ if and only if $\lre(S)=\lre(T)$, and give $\mathcal{AT}/\sim$ the
quotient topology.

\begin{theorem}\cite{br-bdy}  \label{bdy}
 The space $\partial \mathcal{F}_n$ is homeomorphic to $\mathcal{AT}/\sim$. In particular, all boundary points of $\FF_n$ are arational trees.
\end{theorem}

Combining the above Theorems we obtain the crucial mixing property we need (we refer the reader to the Introduction
for the definition of a thick hyperbolic ray).

\begin{theorem} \label{mix}
Let $r$ be a thick hyperbolic ray in Outer space and let $r_\infty \in \partial cv_n$ be the limiting $\R-$ tree.
Then $\lre (r_\infty)$ is strongly arational.

In particular if $Q$ is a convex cocompact purely hyperbolic subgroup of $Out(F_n)$ and $r$ is a quasi-geodesic ray in $Q$ starting at $1 \in Q$, then 
its limit $r_\infty$ in the boundary $\partial \FF_n$ of the free factor complex is strongly arational.
\end{theorem}

\begin{proof}  By Theorem \ref{bdy} every point in $\partial \FF_n$  comes from an arational $\R-$tree.
Hence  $r_\infty$ is arational.

Since $r$ is hyperbolic, the metric bundle over $r$ is hyperbolic by definition.
In particular, the bundle satisfies the flaring condition \cite[Proposition 5.8]{mahan-sardar}.
Hence every element of $F_n$ has non-zero translation length on the limiting $\R$-tree $r_\infty$, thus
ruling out alternative (b) of Theorem \ref{aratimpliesindec}.
 It follows from Theorem \ref{aratimpliesindec} that $r_\infty$ is indecomposable free.
It finally follows from Theorem \ref{aratstrong} that $r_\infty$ is strongly arational. Equivalently, $\lre (r_\infty)$ is strongly arational.

Next, suppose that
$Q$ is a convex cocompact purely atoroidal subgroup of $Out(F_n)$ and $r$  a quasi-geodesic ray in $Q$ starting at $1 \in Q$.
By Theorem 4.1 of \cite{dt1}, an orbit of $Q$ is quasi-convex (in the strong symmetric sense).
Then the  limit point  $r_\infty$ of $r$ lies in $\partial \FF_n$  since the orbit map from $Q$ to $\FF_n$ is a qi-embedding
and is therefore  arational.
Since  $Q$ is purely atoroidal quasi-convex, it follows from Corollary 5.3 of \cite{kdt} that the lamination dual
to the tree $r_\infty$
cannot be carried by a surface with a puncture thus ruling out Case (b) of 
Theorem \ref{aratimpliesindec}. Hence it  is indecomposable free by  Theorem \ref{aratimpliesindec}.
Again, from Theorem \ref{aratstrong} $r_\infty$ is strongly arational.\end{proof}

\subsection{Quasi-convexity}

\begin{theorem} \label{free-coco}  
Let 
$$1 \to H \to G \to Q \to 1$$ 
be an exact sequence of hyperbolic groups with $H=F_n$ and $Q$ convex cocompact. 
Let $K$ be a finitely generated infinite index subgroup of $H$. Then  $K$ 
is quasi-convex  in $G$.
\end{theorem}

\begin{proof}  
We first note that  for each $z \in \partial Q \subset \partial \FF_n$
(where we identify the boundary of $Q$ with the boundary of its orbit in $\FF_n$), the tree $T_z$  is strongly arational
by Theorem \ref{mix}.
In particular, no leaf of $\lre (T_z)$ is carried by $K$.
Hence for all $z  \in \partial Q$, no leaf of $\lre (T_z)$ is carried by $K$.
By Theorem \ref{kdtthm}, the algebraic ending lamination $\overline{\lel^z} = \lre (T_z)$. Further
by Theorem \ref{ct=el}, 
$$
\lct (H,G) = \lel (H,G) = \cup_z \overline{\lel^z}.
$$ 
Hence no leaf of $\lct (H,G)$ is carried by $K$. By Lemma \ref{qccrit} and Remark \ref{subct},  $K$ 
is quasi-convex in $G$.  \end{proof}

\begin{theorem} \label{free-ray} Let $H=F_n$, let $r$ be a thick hyperbolic ray in Outer space $cv_n$ and let $r_\infty \in \partial cv_n$ be the limiting $\R-$ tree.
Let $X$ denote the universal metric bundle over $r$. Let $K$ be a finitely generated infinite index subgroup of $H$. Then any orbit of $K$ in $X$
is quasi-convex.
\end{theorem}

\begin{proof} As in the proof of Theorem \ref{free-coco},
	 the tree $T=r_\infty$  is strongly arational
	 by Theorem \ref{mix}. Hence no leaf of $\lre (T)$ is carried by $K$.
By Proposition \ref{lregel} and Remark \ref{lregelrmk}, the algebraic ending lamination 
$$
\lel (H, X)=\lel \subset \lre (T).
$$ 
Further by Theorem \ref{ct=el} and Proposition \ref{ct=elrmk}, $\lct (H,X) = \lel (H,X)$. Hence no leaf of $\lct (H,X)$ is carried by $K$. By Lemma \ref{qccrit}  and Remark \ref{subct}, any orbit of $K$ in $X$
is quasi-convex in $X$.  \end{proof}

\section{Punctured Surfaces} \label{punct}  
For the purposes of this section $S^h$ is a non-compact finite volume hyperbolic surface and
$H = \pi_1(S^h)$.

\subsection{Quasi-convexity for rays}

\begin{theorem} \label{punct-ray}
Let $r \subset  Teich(S^h)$ be a thick geodesic ray and $r_\infty \in \partial Teich(S^h)$ be the limiting surface ending lamination.
Let $X$ denote the universal metric bundle over $r$ minus a small neighborhood of the cusps and let $\HH$ denote the horosphere
boundary components. Let $K$ be a finitely generated infinite index subgroup of $H$. Then any orbit of $K$ in $X$  is relatively
quasi-convex in $(X, \HH)$.
\end{theorem}

\begin{proof} The proof is slightly more involved than Theorem \ref{surf-ray}; the difficulty arising from the difference
	between the closed diagonal closure
	and the diagonal closure of a geodesic surface lamination $\LL$. Recall that $\LL^d$ denotes the
	{\it closure} of the diagonal closure of $\LL$.
	
	First, observe that if $K$ corresponds to a parabolic subgroup, it is automatically relatively
	quasi-convex. Hence assume that $K$ is not a parabolic subgroup.
By Theorem \ref{mix-surf} (which, recall, holds for punctured surfaces), the lamination $r_\infty$  is strongly arational. Hence no leaf or diagonal of $r_\infty$ is carried by $K$.
By Theorem \ref{end1},  (which, recall, holds for punctured surfaces as well), the Cannon-Thurston lamination $\lct (H, X)= \lel (r_\infty)^d$. 

However, for a punctured surface, $\lel (r_\infty)^d $ does not equal the diagonal closure
of $\lel(r_\infty) $ unlike 
the closed surface case. We shall analyze the difference shortly. Note also that $r_\infty$ refers to the surface
geodesic lamination living in $S^h$, whereas  $\lel (r_\infty)$ refers to collections of pairs of points in $\partial^2_r H$
(or equivalently collections of bi-infinite geodesics in the universal cover of $S^h$).
Let $Diag (\lel (r_\infty))$ denote the diagonal closure of $\lel (r_\infty)$. Let $L_i =
\langle a_i \rangle,$ $ i= 1, \cdots, l$ denote
the peripheral cyclic subgroups of $H$, generated by $a_i$, $i= 1, \cdots, l$ respectively. Let $z_1, \cdots, z_l$
denote the corresponding punctures on $S^h$. There exists a unique connected component $D_i \subset S^h \setminus \lel (r_\infty)$ containing $z_i$. Such a $D_i$ is called a {\it crown domain}.
The boundary components of $D_i$ are finitely many leaves of $r_\infty$. In the universal cover each lift
$\til D_i$ of $D_i$
is an infinite sided polygon stabilized by a conjugate of $a_i$. Any ideal point of such a $\til D_i$
shall be referred to as a {\it crown-tip}.
With this terminology, $Diag (r_\infty)$ consists of pairs of points $(p,q)$ 
such that 

\begin{enumerate}
	\item either $(p,q)$ are end-points of a leaf of $r_\infty$ lifted to the universal cover of $S^h$;
	\item or there is a (necessarily finite) sequence $p=p_1, \cdots, p_n =q$ of crown-tips
	such that $(p_i,p_{i+1})$ are end-points of a leaf of $r_\infty$ lifted to the universal cover.
\end{enumerate}

Next,
 $\lel (r_\infty)^d \setminus Diag (r_\infty)$ consists precisely
  of pairs of points $(p,q)$ in $\partial_r H (=S^1)$ such that $p$
is a fixed point of a conjugate $a_i^g$ of some parabolic $a_i$ and $q$ is a crown-tip on the boundary of the lift 
$\til D_i$ stabilized by $a_i^g$. It follows that any bi-infinite geodesic in $\lel (r_\infty)^d \setminus Diag (r_\infty)$
necessarily has one direction (the direction converging to the crown-tip) asymptotic to a lift of a 
leaf of $r_\infty$. Since
$K$ is finitely generated, it is necessarily relatively quasi-convex in $H$. Hence, if $K$ carries a leaf of 
$\lel (r_\infty)^d \setminus Diag (r_\infty)$, we can translate such a leaf by larger and larger non-parabolic
elements of $K$ and pass to a limit to obtain a leaf of $Diag (r_\infty)$ carried by $K$. This contradicts
the fact (Theorem \ref{mix-surf}) that $r_\infty$ (and hence $Diag (r_\infty)$) is strongly arational. 
It follows that no leaf of $\lct (H,X)$ is carried by $K$. By Lemma \ref{qccritrh}
 and Remark \ref{subct}, any orbit of $K$ in $X$
is relatively quasi-convex in $(X, \HH)$.  \end{proof}

\subsection{Quasi-convexity for Exact sequences}
Let 
$$1 \to H \to G \to Q \to 1$$ 
be an exact sequence of relatively
 hyperbolic groups with $H=\pi_1(S^h)$ for a finite volume hyperbolic surface
$S^h$ with finitely many peripheral subgroups $H_1, \cdots , H_n$.  Here $Q$ is
a convex cocompact subgroup of $MCG(S^h)$, where $MCG$ is taken to be
the pure mapping class group, fixing peripheral subgroups (this is a technical point and is used only for expository convenience).  Note that the normalizer $N_G(H_i)$ is then 
isomorphic to
$H_i \times Q (\subset G)$. The following characterizes convex cocompactness

\begin{prop}\cite[Proposition 5.17]{mahan-sardar}
Let $H = \pi_1(S^h)$ be the fundamental group of a surface with
finitely many punctures and let  $H_1, \cdots ,H_n$ be its peripheral subgroups. Let $Q$ be a
convex cocompact subgroup of the pure mapping class group of $S^h$. Let
$$1 \to H\to G \to
Q  \to 1$$
and
$$1 \to H_i \to N_G(H_i) \to
Q \to 1$$
be the induced short exact sequences of groups. Then $G$ is strongly hyperbolic relative
to the collection  $ \{N_G(H_i)\}, i = 1, \cdots , n$.

Conversely, if $G$ is (strongly) hyperbolic relative to the collection   $ \{N_G(H_i)\}, i = 1, \cdots , n$ then $Q$ is convex-cocompact.\label{punctcoco} \end{prop}

 Since $Q$ is convex cocompact, its orbits in both $Teich(S^h)$ and $CC(S^h)$ are quasi-convex and qi-embedded \cite{farb-mosher, kl, ham-cc} . 
Identify $\Gamma_Q$ with a subset of   $Teich(S^h)$ by identifying the vertices of $\Gamma_Q$ with an orbit $Q.o$ of $Q$ and edges with geodesic segments joining the
corresponding vertices.    

 Let $X_0$ be the  universal curve over $\Gamma_Q$. Let $X_1$ denote $X_0$  with a small neighborhood of the cusps removed. 
Then $X_1$ is a union $\cup_{q \in \partial \Gamma_Q} X_q$, where $X_q$ is a bundle over the quasi-geodesic $[1,q) (\subset \Gamma_Q \subset Teich(S^h))$
with fibers hyperbolic surfaces diffeomorphic to $S^h$  with a small neighborhood of 
the cusps removed. Minsky proved  that 
\begin{enumerate}
\item the quasi-geodesic $[1,q)$ stays 
a bounded distance from a geodesic in \Teich space ending at the point $q \in \partial Teich(S^h))$ \cite{minsky-bddgeom};
\item   $X_q$ is (uniformly) bi-Lipschitz homeomorphic to the convex core minus (a small neighborhood of) cusps of the unique simply degenerate hyperbolic 3-manifold $M$ with conformal structure on the geometrically finite end given
by $o = 1.o \in Teich(S^h)$
and ending lamination of the simply degenerate end given by  $\lel (q)$  \cite{minsky-jams}. 
\end{enumerate}

The convex core of $M$ is denoted by $Y_{q0}$.
 Let $Y_{q1}$ denote $Y_{q0}$  with a small neighborhood of the cusps removed. Thus $X_q, Y_{q1}$ are  (uniformly) bi-Lipschitz homeomorphic.
Let $\til X_q$ denote the universal cover of $X_q$ and $\HH_q$ its collection of boundary horospheres. Then $\til X_q$ is (strongly) hyperbolic relative to $\HH_q$.
 Let $H  =\pi_1(S^h)$ be thought of as a relatively hyperbolic group, hyperbolic relative to the cusp subgroups
  $ \{H_i\}, i = 1, \cdots , n$. The relative hyperbolic (or Bowditch) boundary $\partial_r H$
of the relatively hyperbolic group
is still the circle (as when $S$ is closed) and $\partial_r^2 H$ is defined as $(\partial_r H \times \partial_r H \setminus \Delta)$ as usual.
The existence of a Cannon-Thurston map in this setting from the relative hyperbolic boundary of $H$ to the relative hyperbolic boundary of $(\til{ X_q}, \HH_q)$
has been proved in \cite{bowditch-ct, brahma-bddgeo}. Also, it is established in \cite{bowditch-ct, mahan-elct}
(see Theorem \ref{end1}) that the Cannon-Thurston lamination for the pairs $H$, $\til{ X_q}$
is given by
$$\lct (H, \til{ X_q}) =\lel(q)^d,$$  where $\lel(q)^d$ denotes the closure of the
diagonal closure of   the  ending lamination   $\lel (q)$.

Next, by Proposition \ref{punctcoco} $G$ is strongly hyperbolic relative
to the collection  $ \{N_G(H_i)\}, i = 1, \cdots , n$. Note that the inclusion of $H$ into $G$ is strictly type-preserving as an inclusion of relatively
hyperbolic groups.  The existence of a Cannon-Thurston map for the pair $(H,G)$ is established in \cite{pal-thesis}.

 We shall require a generalization of Theorem \ref{end2} to punctured surfaces to obtain a  description of $\lct(H,G)$.
The description of   the Cannon-Thurston lamination $\lct(H,G)$ for the pair  $(H,G)$  can now be culled from  \cite{mitra-endlam} and   \cite{pal-thesis}. We shall give a brief sketch of the modifications necessary
to the arguments of \cite{mitra-endlam} so as to make them work in the present context.
The crucial technical tool is the construction of a ladder, which we sketch now.
As usual, fix finite generating sets of $H, G$ such that the generating set of $G$ contains that of $H$, thus giving
a natural inclusion of Cayley graphs, $\Gamma_H$ into $\Gamma_G$. Let $\Gamma_Q$ denote the Cayley graph of $Q$ with respect
to the generating set given by the nontrivial elements of the quotient of the generating set of $G$.

Pal \cite{pal-thesis} proves the existence of a qi-section 
$\sigma: \Gamma_Q \to \Gamma_G$. 
Given any $a, b \in H$, we now look at  geodesic segments  $\lambda_q = [a\sigma(q), b\sigma(q)]$ 
in the coset $\Gamma_H\sigma(q) = \sigma(q)\Gamma_H$
   (the equality of left and right cosets follows from normality of  $H$) 
   joining $a\sigma(q), b\sigma(q)$. Note that these are
geodesics in the {\it intrinsic path-metric} on $\sigma(q)\Gamma_H$, which is isometric to $\Gamma_H$.
The union $\bigcup_{q \in Q} \lambda_q$ is called a {\bf ladder} corresponding to $[a,b] \subset \Gamma_H$.
Note here that the ladder construction in \cite{mitra-endlam}  does not require hyperbolicity of $G$
but only that of $H$. Since $H$ is free in the present case, the construction of the ladder  goes through.  

As in \cite{mitra-endlam} (see also Definition \ref{lelz} and
Proposition \ref{ct=elrmk}), we assign to every boundary point $z \in \partial Q$,
an algebraic ending lamination $\overline{\lel^z}$. Similarly (as in the hyperbolic case), for every $z$, there is a Cannon-Thurston lamination
$\lct (z)$. 
The proof of the description of the ending lamination in  \cite{mitra-endlam} (using the ladder)
now shows that
the Cannon-Thurston lamination $\lct(H,G)$ for the pair  $(H,G)$ is the closure of the transitive 
closure of the union $\cup_{z \in \partial Q}\lct (z)$ (cf.\ \cite[Section 4.4]{mahan-elct}).
 We elaborate on this a bit.
Recall that  $X_1$ is a union $\cup_{q \in \partial \Gamma_Q} X_q$, and that the universal cover of $X_1$ is naturally quasi-isometric to $G$.
Thus $\Gamma_G$ can be thought of as a union (non-disjoint) of the metric bundles over $[1,q), $ as $q$ ranges over $\partial Q$. 
In fact if $P: G \to Q$ denotes projection, then $\til X_q$ is quasi-isometric to $P^{-1} ([1,q))$.
The construction of the ladder
and a coarse Lipschitz retract of $\Gamma_G$ onto it then shows that a leaf of the  Cannon-Thurston lamination $\lct(H,G)$ arises as a concatenation of at most two infinite rays, each 
of which lies in a leaf of the  Cannon-Thurston lamination $\lct(H,P^{-1} ([1,q)))$ for some $q$.  Thus $ \lct (H,G)$ is the closure of the transitive closure of the union 
$\cup_{q \in \partial Q}\lct (H, \til{ X_q})$, i.e.\
 $$\lct (H,G) = (\cup_{z \in \partial Q} \lct (z))^d.$$

We need to show that $\overline{\lel^z}=\lct (z)$.
 Lemma 3.5 of \cite{mitra-endlam} goes through verbatim to show that $\overline{\lel^z} \subset\lct (z)$.
It remains to show that $\lct (z) \subset \overline{\lel^z}$. But this is exactly the content of the main theorem of
\cite{bowditch-ct} (see also \cite{brahma-bddgeo}).

 We combine all this in the following.

\begin{theorem} \cite{minsky-jams, bowditch-ct, mitra-endlam, mahan-sardar}
Let $H = \pi_1(S^h)$ be the fundamental group of a surface with finitely many 
punctures and let  $H_1, \cdots ,H_n$ be its peripheral subgroups. Let $Q$ be 
a convex cocompact subgroup of the pure mapping class group of $S^h$. Let
$$1 \to H\to G \to
Q  \to 1$$
and
$$1 \to H_i \to N_G(H_i) \to
Q \to 1$$
be the induced short exact sequences of groups. Then $G$ is strongly hyperbolic relative
to the collection  $ \{N_G(H_i)\}, i = 1, \cdots , n$.

 Further, $\lct (H,G) = (\cup_{z \in \partial Q} \lel (z))^d$. \label{end3} \end{theorem}

We can now prove our last quasi-convexity Theorem:

\begin{theorem} \label{punct-coco}
Let $H$ and $G$ be as in Theorem~\ref{end3} and let $K$ be a finitely generated 
infinite index subgroup of $H$. Then $K$ is relatively quasi-convex  in $G$.
\end{theorem}

\begin{proof} Without loss of generality, $K$ is not contained in a parabolic subgroup of $H$ (since then there
	is nothing to prove). As in the proof of Theorem \ref{punct-ray} above, the lamination ${\lel} (q)$  
is strongly arational for each $q \in \partial Q $. Hence, as in the proof of Theorem \ref{punct-ray},
 no leaf 
of $\lel (q)^d$ is carried by $K$ as $q$ ranges over $ \partial Q$. By Theorem \ref{end3}, $\lct (H,G) $ is the
closure of the transitive 
closure of  $ \cup_{z \in \partial Q} \lel (z)^d$. It follows (again as in the proof of Theorem \ref{punct-ray})
that no leaf of $\lct (H,G)$ is carried by 
$K$. By Lemma \ref{qccritrh}  and Remark \ref{subct}, $K$ is relatively quasi-convex in $G$.  \end{proof}

\section*{Acknowledgments} The authors are grateful to Ilya Kapovich and Mladen Bestvina for helpful correspondence. We also thank Ilya Kapovich for telling us of the paper
\cite{kdt} and Spencer Dowdall for telling us about \cite{dt2}. We thank Ilya Kapovich and Samuel Taylor for pointing out gaps in an earlier version.
This work was done during a visit of the first author to University of Toronto. He is grateful to the
 University of Toronto for its hospitality during the period. We are grateful to the anonymous referee for a thorough inspection of the paper
 and many helpful comments.
\bibliography{elqc}
\bibliographystyle{alpha}

\end{document}